\documentclass[reqno,11pt]{amsart}
\usepackage{amsmath,amssymb,mathrsfs,amsthm,amsfonts}
\usepackage[inline]{enumitem}
\usepackage[usenames,dvipsnames]{xcolor}
\usepackage{hyperref}
\usepackage[percent]{overpic}
\usepackage{comment}
\usepackage{stmaryrd}
\usepackage{algorithm}
\usepackage{algorithmic}
\usepackage{pgfplots}
\usepackage{subcaption}
\usepackage{tikz}
\usepackage{bm}
\usetikzlibrary{calc}
\usetikzlibrary{arrows.meta,backgrounds}
\usepackage{multirow,array,longtable,booktabs}
\usetikzlibrary{arrows}

\hypersetup{
	colorlinks=true, linkcolor=blue,
	citecolor=ForestGreen
}
\usepackage[paper=letterpaper,margin=1in]{geometry}
\DeclareMathAlphabet{\mathpzc}{OT1}{pzc}{m}{it}

\newtheorem{theorem}{Theorem}[section]
\newtheorem{lemma}[theorem]{Lemma}
\newtheorem{proposition}[theorem]{Proposition}
\newtheorem{corollary}[theorem]{Corollary}
\theoremstyle{definition}
\newtheorem{definition}[theorem]{Definition}

\newtheorem{remark}[theorem]{Remark}
\numberwithin{equation}{section}
\allowdisplaybreaks
\usepackage{acronym}

\newcommand{\be}{\begin{equation}}
\newcommand{\ee}{\end{equation}}

\newcommand{\ind}{\mathbf{1}}	

\makeatletter

\newcommand{\R}{\mathbb{R}} 

\newcommand{\e}{\varepsilon}

\newcommand{\red}{\textcolor{red}}

\usepackage{graphicx}

\title[Martingale Optimal Transport and Schr\"odinger Bridges]{Martingale Optimal Transport and Martingale Schrödinger Bridges for Calibration of Stochastic Volatility Models}

\author[A.\ Zitridis]{Antonios Zitridis}
\address{A.\ Zitridis,
	Department of Mathematics, University of Michigan,\newline
    \null\hspace{2.2cm}
    Ann Arbor, MI, 48109, United States.}
\email{zitridis@umich.edu}

\begin{document}
	\begin{abstract}
Motivated by recent developments in the calibration of stochastic volatility models (SVMs for short), we study continuous-time formulations of martingale optimal transport and martingale Schr\"odinger bridge problems. We establish duality formulas and also provide alternative proofs, via different techniques, of duality results previously established in the mathematical finance literature. Applications include calibration of SVMs to SPX options, as well as joint calibration to both SPX and VIX options.
	\end{abstract}
	
	
	\maketitle
	{
		}


\vspace{-9mm}

\section{Introduction}

\noindent
Our aim in this paper is to study Martingale Optimal Transport and Martingale Schrödinger Bridges, with a particular motivation coming from the calibration of stochastic volatility models. We establish duality formulas and also provide alternative proofs, using different techniques, for results that have previously appeared in the mathematical finance literature. Our approach relies on techniques that are commonly employed in mean field game theory (see \cite{lasry2007mean, huang2006large} and subsequent works). Our focus is theoretical, and we do not address numerical approaches or their efficiency.

\subsection{Motivation and Overview}

\noindent
Since the introduction of the Black–Scholes model \cite{black1973pricing}, a significant amount of effort has been devoted to developing sophisticated volatility models that properly capture the market dynamics. To correct for any systematic errors and ensure accurate predictions, it is necessary to properly adjust these models to Vanilla options\footnote{The term Vanilla options refers to financial contracts that give the holder the right to buy or sell an asset at a predetermined price within a specific time frame.}, such as an S\&P 500 (SPX) option, which depends on the eponymous stock market index\footnote{The Standard and Poor's 500 index, or simply the S\&P 500, tracks the performance of 500 of the largest publicly traded companies in the United States.}, and VIX options (options on the volatility index). These problems are called calibration problems and are the subject of intense ongoing research; we refer to the survey \cite{guo2021optimal} and the references therein. 
\vspace{1.5mm}

\noindent
To calibrate stochastic volatility models, many authors have employed the tools of optimal transport and Schr\"odinger bridges, which are related to the transportation of a probability measure to another in an ``optimal'' way. In these cases, the optimization is with respect to a given cost functional and the relative entropy, respectively. Moreover, due to the arbitrage-free nature of many studies, additional martingale constraints are considered. 
\vspace{1.5mm}

\noindent
We begin by giving an overview of the calibration of stochastic volatility models to Vanilla options and, after that, the role of Martingale Optimal Transport and Martingale Schrödinger Bridges in this context.


\vspace{3mm}

\noindent
Let $\mu_1,\mu_2\in\mathcal{P}(\R)$ be two Borel probability measures with finite second moments and in convex order\footnote{This means that $\int h(x)d\mu_1(x)\leq \int h(x)d\mu_2(x)$ for any convex function $h:\R\rightarrow \R$.} $\mu_1\leq_c \mu_2$, and $T_2>T_1> 0$. In the filtered probability space $(\Omega,\mathcal{F},\mathbb{F}=(\mathcal{F}_t)_{t\ge 0},\mathbb{P}_0)$, with $\Omega:=C([0,T_2];\R^2)$, we consider the continuous time stochastic volatility model 
\be\label{SVM}
\begin{cases}
    dX_t=\sigma(X_t,Y_t)dW_t,\\
    dY_t= b(X_t,Y_t)dt+\tau_1(X_t,Y_t)dW_t+\tau_2(X_t,Y_t)dW_t^{\perp},
\end{cases}
\ee
where $W,W^{\perp}$ are independent $\mathbb{P}_0$-Brownian motions and $\sigma,b,\tau_1,\tau_2:\R^2\rightarrow \R$. We also assume that $\sigma(x,y)=x\tilde{\sigma}(y)$ for some other $\tilde{\sigma}:\R\rightarrow \R$. There are many known stochastic volatility models that appear in the mathematical finance literature captured by this general form, such as the Heston model and the SABR model. In \eqref{SVM}, $X_t$ represent the price of S\&P500, while $Y_t$ is the value of its volatility. We let $\mathcal{P}$ be the space of probability measures over $(\Omega,\mathcal{F},\mathbb{F}=(\mathcal{F}_t)_{t\ge 0})$. Calibrating \eqref{SVM} to SPX options, is equivalent to finding a $\mathbb{P}\in\mathcal{P}$, such that the laws $\mathcal{L}^{\mathbb{P}}(X_{T_1})=\mu_1$ and $\mathcal{L}^{\mathbb{P}}(X_{T_2})=\mu_2$\footnote{In this case $T_1,T_2$ are the maturities of the options, while the constraints on the law give the prices of the call options with maturities $T_1,T_2$:  $\mathbb{E}^{\mathbb{P}}[(X_{T_i}-K)_+]=\int (x-K)_+d\mu_i(x)$, $i=1,2$, for all strike prices $K\in \R$.}. At the same time, we require $X_t$ to be an $((\mathcal{F}_t),\mathbb{P})$-martingale and $\mathbb{P}$ to be close to $\mathbb{P}_0$, in the appropriate sense, so that the models remain as close as possible to the initial belief \eqref{SVM}. 
\vspace{1mm}

\noindent
In many works, the role of martingale optimal transport and martingale Schr\"odinger bridges is crucial, especially in quantifying what is the optimal $\mathbb{P}$, or how close $\mathbb{P}$ will be to $\mathbb{P}_0$; depending on whether there is an initial belief $\mathbb{P}_0$. In particular, in the former case, $\mathbb{P}$ is chosen in the sense of martingale optimal transport, while in the latter case, we establish the probability measure $\mathbb{P}$ that minimizes the relative entropy between $\mathbb{P}$ and $\mathbb{P}_0$. These minimization problems often admit duality formulas, which allow for numerical approximations of the optimal $\mathbb{P}^*$, whenever it exists. We refer to \cite{guo2017local,guo2022joint,guyon2022dispersion} for numerical simulations using the duality formulas.
\vspace{2mm}

\subsection{Calibration via Martingale Optimal Transport}
We start with the setup for the Martingale Optimal Transport problem (in continuous time). In terms of the discussion from the previous subsection, this corresponds to calibration when there is no initial belief $\mathbb{P}_0$, therefore our problem in this case is model-independent.
\vspace{1.5mm}

\noindent
Let $T_0\in [0,T_2]$, $0<T_1<T_2$, $\mu_0\in \mathcal{P}(\R)$ and $L:\R\rightarrow \R$ a given cost function. The Martingale Optimal Transport problem we will be studying is the minimization problem
\be\label{mot}
V(T_0,\mu_0,\mu_1,\mu_2):=\inf_{\mathbb{P}\in \mathcal{P}_{m}}\bigg\{ \mathbb{E}^{\mathbb{P}}\left[\int_{T_0}^{T_2} L(\sigma_t^2)dt\right]\bigg\}.
\ee
The minimization is over the subset the set of probability measures $\mathcal{P}$:
\be \label{mcondition}
\begin{split}
    \mathcal{P}_{m}:=\bigg\{ \mathbb{P}\in\mathcal{P}\bigg| X_t \text{ is an } ((\mathcal{F}_t),&\mathbb{P})\text{-martingale,}\\[-12pt]
    &\text{ a semimartingale, and } \mathcal{L}^{\mathbb{P}}(X_{T_i})=\mu_i,\;i=0,1,2\bigg\}.
\end{split}
\ee
$X_t$, being a martingale and $\mathbb{P}$-semimartingale, satisfies $dX_t=\sigma_tdB_t$ with $(B_t)_{t\ge 0}$ being a $1$-dimensional Brownian motion over the filtered probability space $(\Omega, \mathcal{F},\mathbb{F},\mathbb{P})$ and $\sigma_t$ is an adapted process with values in $\R$. Note that if $T_0>T_1$, then the constraint $\mathcal{L}^\mathbb{P}(X_{T_1})$ becomes irrelevant for \eqref{mot}. This is the case of  continuous time martingale optimal transport studied in \cite{huesmann2019benamou}.
\vspace{2mm}

\noindent
It is worth mentioning a brief history of \eqref{mot}. Martingale optimal transport originates from the classical Monge-Kantorovich optimal transport problem \cite{monge1781memoire, kantorovich1948problem}. In discrete time and in the case where $T_1<T_0$, this problem can be stated as
\be \label{ot}
\inf\left\{ \mathbb{E}[c(X_{T_0},X_{T_2})]\; | \mathcal{L}^{\mathbb{P}}(X_{T_0})= \mu_0,\; \mathcal{L}^{\mathbb{P}}(X_{T_2})=\mu_2\right\},
\ee
where $c:\R^{2}\rightarrow \R^+$ a given function and where the minimization is over all random vectors $(X_0,X_{T_2})$ in a probability space $(\Omega,\mathcal{F},\mathbb{P})$, which are called transportation plans. For every transportation plan and $\omega\in \Omega$, the mass from position $X_0(\omega)$ is transported to position $X_{T_2}(\omega)$. In 2000, Benamou and Brenier \cite{benamou2000computational} introduced a continuous in time formulation of this problem. The dual formulation of the continuous in time problem has been extensively studied as an application of the Fenchel-Rockafellar theorem \cite[Theorem 1.9]{villani2021topics}.
\vspace{1.5mm}

\noindent
Extending the Monge-Kantorovich problem and its Benamou-Brenier formulation, Mikami and Thieullen  \cite{mikami2006duality}, and Tan and Touzi \cite{tan2013optimal} introduced the continuous time semi-martingale optimal transport problem and its dual formulation. Roughly speaking, this problem has the form \eqref{mot} when $T_1<T_0$ and when $\mathcal{P}_m$ from \eqref{mcondition} does not include the martingale condition for $X_t$. In this case $X_t$ might have a drift and, hence, satisfies $dX_t=\beta_tdt+\sigma_tdB_t$ for some adapted process $\beta_t$ with values in $\R$; $L$ might also depend on $\beta_t$ in a more general setup.
\vspace{1.5mm}

\noindent
The continuous in time martingale optimal transport problem \eqref{mot} or its discrete version (\eqref{ot} with the additional condition that $X$ is a discrete martingale) has many applications in finance, especially for questions of worst case bounds for derivative prices in model-independent finance, see e.g \cite{galichon2014stochastic}. 
We note that, unlike the Monge-Kantorovich optimal transport problem and its Benamou-Brenier formulation, the connection between the continuous time martingale optimal transport problem and its discrete analogue
is still not well understood (see \cite{huesmann2019benamou}). However, \eqref{mot} still posesses a duality formulation, which is useful in approximating the optimal $\mathbb{P}$ (see e.g. \cite{guo2017local, guo2022joint} among others) and is the first result covered in this paper.

\vspace{1mm}

\subsection{Calibration via Martingale Schr\"odinger Bridges}
Martingale Schr\"odinger bridges were introduced by Labordère \cite{henry2019martingale} and originate from the classical Schr\"odinger bridge problem \cite{schrodinger1931umkehrung,schrodinger1932theorie}, which is related to minimizing the relative entropy\footnote{The relative entropy between two measures $\mu,\nu$ is defined as follows.
$H(\mu|\nu)=\begin{cases}
    \int \log\frac{d\mu}{d\nu}d\mu, & \text{if } \mu\ll \nu,\\
    +\infty, & \text{ otherwise.}
\end{cases}$} $H$ between two probability measures (we refer to the survey \cite{leonard2013survey}).

\vspace{2mm}

\noindent
We consider $t_0\in [0,T]$ and \eqref{SVM} with initial condition $(X_{t_0},Y_{t_0})=(x,y)\in \R^2$, for some $x,y>0$.
The martingale Schr\"odinger bridge is the measure that is closest to $\mathbb{P}_0$ in the sense of relative entropy, while the SVM under $\mathbb{P}$ will satisfy some given conditions including the martingality of $X_t$. In the notation we introduced at the beginning of the section, the martingale Schr\"odinger bridge problem that calibrates \eqref{SVM} is the minimization problem
\be\label{minimization}
D_{\mathbb{P}_0}:=\inf _{\mathbb{P}\in\mathcal{C}_{t_0}(\mu_1,\mu_2)}H(\mathbb{P}|\mathbb{P}_0),
\ee
where $H$ is the relative entropy between $\mathbb{P}$ and $\mathbb{P}_0$ and
\be\label{condition}
\mathcal{C}_{t_0}(\mu_1,\mu_2)=\begin{cases}
    \bigg\{\mathbb{P}\in \mathcal{P}\bigg| \mathcal{L}^{\mathbb{P}}(X_{T_i})=\mu_i,\; i=1,2\text{ and } X\text{ is an } ((\mathcal{F}_t),\mathbb{P})\text{-martingale}\bigg\},&\text{if }t_0\in [0,T_1),\\
    \bigg\{\mathbb{P}\in\mathcal{P}\bigg| \mathcal{L}^{\mathbb{P}}(X_{T_2})=\mu_2,\; X\text{ is an } ((\mathcal{F}_t),\mathbb{P})\text{-martingale}\bigg\},& \text{if }t_0\in [T_1,T_2].
    \end{cases}
\ee

\vspace{1mm}

\noindent
Apart from calibrating SVM to SPX options \cite{henry2019martingale}, martingale Schr\"odinger bridges were also recently used to solve the longstanding \textit{joint S\&P 500 - VIX smile calibration problem} \cite{guyon2022dispersion,guyon2020joint} (we refer to these papers as well as their cited references for numerical simulations and different approaches to this problem). In this problem, a constraint on the prices of VIX options (VIX constraint) is included in \eqref{minimization} yielding the minimization problem
\be\label{minvix}
\tilde{D}_{\mathbb{P}_0}:=\inf_{\mathbb{P}\in \mathcal{C}_{t_0}(\mu_1,\mu_2), \; \text{VIX}} H(\mathbb{P}|\mathbb{P}_0),
\ee
where $\mathcal{C}_{t_0}(\mu_1,\mu_2)$ was defined in \eqref{condition}, $(X_t,Y_t)$ satisfy \eqref{SVM}. The VIX is a constraint on the law of
 \be\label{vix initial}
 V_{\mathbb{P}}=\mathbb{E}^{\mathbb{P}}\left[ -\log(X_{T_2})+\log(X_{T_1})|\mathcal{F}_{T_1}\right],\;\; \mathcal{L}^{\mathbb{P}}(V_{\mathbb{P}})=\mu_3,
 \ee
for some given probability distribution $\mu_3\in\mathcal{P}(\R)$ with support in $[0,+\infty)$. The minimizer of \eqref{minvix}, whenever it exists, is called {\it dispersion constrained martingale Schr\"odinger bridge}.
\vspace{1mm}

\begin{remark}\label{VIXrem}
The name VIX constraint originates from the VIX index: $\text{VIX}_{\mathbb{P}}=100\sqrt{\frac{2}{T_2-T_1}V_{\mathbb{P}}}$ (typically $T_2-T_1=30$ days), where $\mathbb{P}\in \mathcal{C}_0(\mu_1,\mu_2)$. The CBOE Volatility Index (VIX index) is a popular measure of the stock market's expectation of volatility based on SPX options. Clearly, since $-\log$ is convex and $X_t$ is a martingale, we have $V_{\mathbb{P}}\ge 0$. Therefore, knowing the distribution of $\text{VIX}_{\mathbb{P}}$, means that we also know the distribution of $V_{\mathbb{P}}$. This allows us to consider \eqref{vix initial} instead of a constraint on the law of $\text{VIX}_{\mathbb{P}}$ in the calibration with respect to VIX options.
\end{remark}

\noindent
In the above references, if we assume that there are no further constraints, the numerical construction of the martingale Schr\"odinger bridge or the dispersion constrained martingale Schr\"odinger bridge, as in martingale optimal transport, again relies on the dual formulations of the problems. The second aim of the current paper is to give a new proof for a duality formula for \eqref{minimization} and a proof for a duality formula for \eqref{minvix} when we consider a relaxed version of the VIX constraint \eqref{vix initial}.

\subsection{Organization of the paper} For the remainder of section 1 we present the main results of the current paper: Theorems \ref{main2}, \ref{main4}, \ref{main5}, and Corollary \ref{cor1}. In section 2, we give the notation and the assumptions we will be using. We also prove some preliminary results. The proofs of Theorem \ref{main2} and Corollary \ref{cor1} are given in section 3. In section 4, we prove Theorems \ref{main4} and \ref{main5}. In the appendix, we state some important results from the literature that are used in the proofs, we prove some technical propositions and we give a quick overview of the (relevant) theory of viscosity solutions for Hamilton-Jacobi equations with the presence of Dirac delta terms that appear in our results. We note that sections 3 and 4 can be read independently.

\subsection{Main results} Our main results in this paper consist of new proofs to duality formulas for martingale optimal transport or martingale Schr\"odinger bridges already existing in the literature or proofs of such duality formulas in case they were written down as formal statements. Numerical simulations demonstrating the use of these duality formulas for calibrating various stochastic volatility models are either beyond the scope of this paper or can be found in the references.
\vspace{3mm}

\noindent
To illustrate our techniques in a simpler format, we start with the model-independent problem \eqref{mot}. We have the following duality theorem.
\vspace{2mm}

\begin{theorem}\label{main2}
Suppose that (A1), (A2) are true. Let $\mu_0,\mu_1,\mu_2\in \mathcal{P}_2(\R)$ be in convex increasing order ($\mu_0\leq_c \mu_1\leq_c \mu_2$), $0<T_1<T_2$ and $T_0\in [0,T_2]$. Then, \eqref{mot} is admissible. Furthermore, whenever $V(T_0,\mu_0,\mu_1,\mu_2)$ is finite, the following duality formula holds
\be \label{duality}
V(T_0,\mu_0,\mu_1,\mu_2)=\sup_{u_1,u_2\in \text{Lip}}\bigg\{ \int u(T_0,x)d\mu_0(x) -\int u_2(x) d\mu_2(x) -{\bf 1}_{[0,T_1]}(T_0)\int u_1(x)d\mu_1(x)\bigg\},
\ee
where the supremum is taken over all bounded Lipschitz functions $u_1,u_2$, $u$ is a viscosity solution, in the sense of Definition \ref{defvisc}, of the Hamilton-Jacobi equation
\be\label{HJmain}
\begin{cases}
    -\partial_t u+H\left( \frac{D^2u}{2} \right)= \delta_{T_1}(t)u_1,& (t,x)\in [T_0,T_2]\times \R,\\
    u(T_2,x)=u_2(x), & x\in \R,
\end{cases}
\ee
and $H(a)=\sup_{b\ge 0}\{-ab-L(b)\}$, for any $a\in \R$, is the Hamiltonian.
Moreover, the same formula holds if the supremum in \eqref{duality} is taken over $u,u_1,u_2$ such that $u$ is a classical super-solution of \eqref{HJmain}.
\end{theorem}
\vspace{1mm}

\noindent
It is worth noting that a duality formula for calibration via martingale optimal transport was also proved in \cite{guo2022calibration, guo2017local} in the case where the constraints for $X_{T_1},X_{T_2}$ are not on their law, but on their averages. Additionally, in the case $T_0>T_1$, i.e. the intermidiate constraint at time $T_1$ becomes irrelevant, our duality formula can also be found in \cite{huesmann2019benamou} for smooth supersolutions of the Hamilton-Jacobi equation and came as an application of the Fenchel-Rockafellar theorem. Our main idea is to use the Von-Neumann minimax principle (Theorem \ref{Von}), instead of the Fenchel-Rockafellar theorem, as well as ideas from mean field control theory.
\vspace{2mm}

\noindent
As a consequence of Theorem \ref{main2}, we have the following Corollary.

\begin{corollary} \label{cor1}
    Suppose that (A1) and (A2) hold and that $H$ is uniformly elliptic; that is, there exists a constant $\lambda>0$ such that $H'(a)\leq -\lambda $ for every $a\in \R$. Then, the supremum in \eqref{duality} is achieved at $u_1^*,u_2^*\in C^2_b$ if and only if the infimum in \eqref{mot} is achieved when the diffusion of $X_s$ is $\sigma_s=\left(-H'\left(\frac{D^2u(s,X_s)}{2}\right)\right)^{1/2}$, where $u$ satisfies \eqref{HJmain} in the classical sense of Definition \ref{defvisc} for the functions $u_1^*,u_2^*\in C^2_b$.
\end{corollary}

\vspace{2mm}

\noindent
For calibration via martingale Schr\"odinger bridges our first aim is to give a new proof of the duality formula established in \cite[Theorem 3.5]{henry2019martingale}, again by employing the Von-Neumann minimax theorem. Before stating the result, we consider the generator of \eqref{SVM}
\be\label{gen}
\mathcal{L}^0_{x,y}=b(x,y)\partial_y +\frac{\sigma^2(x,y)}{2}\partial_{xx}+\frac{\tau_1^2(x,y)+\tau_2^2(x,y)}{2}\partial_{yy}+\sigma(x,y)\tau_1(x,y)\partial_{xy}.
\ee
\vspace{2mm}

\begin{theorem}\label{main4}
Let $\mu_1,\mu_2\in \mathcal{P}(\R)$. We assume (A3) and that \eqref{minimization} is admissible. Then, for $t_0\in [0,T_2]$ and $(X_t,Y_t)_{t\in [t_0,T_2]}$ satisfying \eqref{SVM} we have
\be\label{dual}
\inf_{\mathbb{P}\in\mathcal{C}_{t_0}(\mu_1,\mu_2)}H(\mathbb{P}|\mathbb{P}_0)=\sup_{u_1,u_2\in Lip}\bigg\{ -{\bf 1}_{[0,T_1]}(t_0)\int u_1d\mu_1-\int u_2d\mu_2+ u(t_0,X_{t_0},Y_{t_0})\bigg\},
    \ee
where the supremum is taken over all bounded and Lipschitz functions $u_1,u_2$, and $u$ is a viscosity solution 
of the Hamilton-Jacobi equation
\be\label{HJS}
\begin{cases}
    -\partial_t u-\mathcal{L}^0_{x,y}u+\frac{1}{2}\tau_2^2(x,y)(\partial_y u)^2=\delta_{T_1}(t)u_1(x),\;t\in [0,T_2],\\
    u(T_2,x,y)=u_2(x),
\end{cases}\ee
in the sense of Definition \ref{defvisc}, where $\mathcal{L}^0$ is as in \eqref{gen}. Furthermore, if $\sigma\in C^{1,1}$ (e.g $\sigma(x,y)=xy$) and  the maximization problem \eqref{dual} admits an optimizer $u_1^*,u_2^*$, then the optimal $\mathbb{P}^*$ in \eqref{dual} is given by
\be\label{optimalP}
\frac{d\mathbb{P}^*}{d\mathbb{P}_0}\bigg|_{\mathcal{F}_{T_2}}=e^{-\int_{t_0}^{T_2}\tau_2(X_t,Y_t)\partial_y u^*(t,X_t,Y_t)dW_t^{\perp}-\frac{1}{2}\int_{t_0}^{T_2}\tau_2^2(X_t,Y_t)|\partial_y u^*(t,X_t,Y_t)|^2dt},
\ee
where $u^*$ the solution of \eqref{HJS} corresponding to $u_1^*,u_2^*$.
\end{theorem}
\vspace{2mm}

\noindent
Finally, our method allows us to provide a rigorous proof of a weaker version of the formal duality formula that appeared (formally) in \cite[p.~6]{guyon2022dispersion} and \cite[Theorem~4.6]{henry2019martingale}, which addresses the joint SPX–VIX calibration problem. In particular, we assume that the VIX constraint for the random variable $V_{\mathbb{P}}$ in \eqref{vix initial} takes the form 
\be\label{VIX'}
\text{VIX'}:=\left\{\mathbb{P}\in\mathcal{P}\bigg| \mathcal{L}^{\mathbb{P}}(V_{\mathbb{P}})\leq_{c,l}\mu_3\right\}
\ee
(instead of equality)\footnote{We write $\mu\leq_{c,l}\nu$, if $\int h(x)d\mu(x)\leq \int h(x)d\nu(x)$ for all convex and lower bounded $h:[0,+\infty)\rightarrow \R$. For example, $h(x)=(x-K)_+$, for any $K\in \R.$}. We name this convex order constraint \textit{VIX' constraint} and we set 
$$D'_{\mathbb{P}_0}=\inf_{\mathcal{C}_{t_0}(\mu_1,\mu_2),\;\text{VIX'}}H(\mathbb{P}|\mathbb{P}_0).$$
\vspace{1mm}

\noindent
To state our result we consider the operator
\be\label{gen2}
\mathcal{L}^0_{w,y}=b(e^w,y)\partial_y-\frac{\tilde{\sigma}^2(y)}{2}\partial_w+\frac{\tilde{\sigma}^2(y)}{2}\partial_{ww}+\frac{\tau_1^2(e^w,y)+\tau_2^2(e^w,y)}{2}\partial_{yy}+\tilde{\sigma}(y)\tau_1(e^w,y)\partial_{wy}.
\ee
This is the generator of the process $(W_t,Y_t)$, where $W_t:=\log(X_t)$ and $(X_t,Y_t)$ satisfies \eqref{SVM}. This change of variable is possible due to our assumption that $X_{t_0}> 0$ and $\sigma(x,y)=x\tilde{\sigma}(y)$, which implies that $X_t> 0$ almost surely.

\begin{theorem}\label{main5}
Assume (A3) and (A4). Suppose $(X_t,Y_t)$ has the dynamics \eqref{SVM} with initial condition $X_{t_0}>0$ and $Y_{t_0}>0$ and that $\mu_1,\mu_2,\mu_3\in \mathcal{P}(\R)$ such that $\mu_1\le_c\mu_2$ and $\int_0^{+\infty}(\log(x))^2d(\mu_1+\mu_2)(x)<\infty$. We further assume that \eqref{minvix} with VIX' constraint is admissible.
For any $\delta\in \R$, there exists a unique viscosity solution $u$ of the state constraint problem
\be\label{HJvix'}
\begin{cases}
    -\partial_tu-\mathcal{L}^0_{x,y}u+\frac{1}{2}\tau_2^2(x,y)(\partial_yu)^2=0,& \delta\in \R,\;t\in(T_1,T_2],\;(x,y)\in (0,+\infty) \times\R,\\
    u(T_2,x,y)=u_2(x)-\delta \log(x), &(x,y)\in (0,+\infty)\times\R,\\
    u(t,0,y)=\text{sign}(\delta)(+\infty), & t\in (T_1,T_2],\; y\in \R.
\end{cases}\ee
such that $|u(t,x,y)|\leq C(1+(\log(x))^2)$ for all $(t,x,y)\in [T_1,T_2]\times (0,+\infty)\times \R$.\\
Let $\Phi(x,y)=\inf_{v\geq 0}\sup_{\delta\in \R}\left\{ u_3(v)-\delta (G(x)+v) +u(T_1^+,x,y;\delta)\right\}$, for $x>0$ and $y\in \R$. Then, $\Phi$ is bounded and, for $t_0<T_1$, we have
\be\label{dual101}
D'_{\mathbb{P}_0}=\inf_{\mathcal{C}_{t_0}(\mu_1,\mu_2), \text{VIX'}}H(\mathbb{P}|\mathbb{P}_0)=\sup_{u_1,u_2,u_3}\bigg\{-\int u_1d\mu_1-\int u_2d\mu_2-\int u_3d\mu_3+u(t_0,X_{t_0},Y_{t_0})\bigg\},
\ee
where the supremum is taken over all bounded and Lipschitz functions $u_1$, $u_2$, and $u_3$ convex, and lower bounded defined on $[0,+\infty)$. In addition, $u(t,x,y)$, $t\in [t_0,T_1]$, $x>0$, $y\in \R$ is the value function of the optimal control problem
\be\label{control}
u(t,x,y):=\inf_{\alpha\in L^2(dm_s\otimes ds)}\bigg\{ \mathbb{E}\left[ \int_{t}^{T_1}\frac{|\alpha(s,X_s,Y_s)|^2}{2}ds+u_1(X_{T_1})+\Phi(X_{T_1},Y_{T_1})\right]\bigg\},
\ee
where $m_s$ is the distribution of the pair $(X_s,Y_s)$ satisfying for $s\in[t,T_1]$
\be \nonumber
\begin{cases}
dX_s=\sigma(X_s,Y_s)dW_s,\\
dY_s=\left( b(X_s,Y_s)+ \tau_2(X_s,Y_s)\alpha(s,X_s,Y_s)\right)ds+\tau_1(X_s,Y_s) dW_s+\tau_2(X_s,Y_s)dW_s^{\perp},\\
X_{t}=x,\;\;Y_{t}=y.
\end{cases}\ee
Furthermore, if $\tau_2$ is assumed to be constant, then $\Phi$ is continuous and $u$ from \eqref{control} can be characterized as follows: for any $t\in [t_0,T_1]$, $x>0$ and $y\in\R$,
\be\label{character}
u(t,x,y)=v(t,\log(x),y),
\ee
where $v:[t_0,T_1]\times\R^2\rightarrow \R$ is the unique viscosity solution of
\be\label{HJvix}
\begin{cases}
   -\partial_tv-\mathcal{L}^0_{w,y}v+\frac{1}{2}\tau_2^2(e^w,y)(\partial_y v)^2=0, &t\in [t_0,T_1],\; (w,y)\in \R^2,\\
   v(T_1,w,y)=u_1(e^w)+\Phi(e^w,y),& (w,y)\in \R^2.\\
\end{cases}\ee
Here, $\mathcal{L}^0_{x,y},\; \mathcal{L}^0_{w,y}$ are given in \eqref{gen}, \eqref{gen2}, respectively.
\end{theorem}
\vspace{1mm}

\begin{remark}
(i) We note that, due to the dynamics \eqref{SVM}, the assumption $\sigma(0,y)=0$ for any $y\in \R$ and the initial condition $X_{t_0}>0$, we have $X_t> 0$ almost surely for any $t\geq t_0$. This allows us to write the formula \eqref{vix initial} and equations \eqref{HJvix'}, \eqref{HJvix}, \eqref{character}.
\vspace{1mm}

\noindent
(ii) In \cite[p.6]{guyon2022dispersion}, the above duality formula \eqref{dual101}, was written formally in the case where $\sigma(x,y)=xy$, $\tau_1(x,y)=\rho \sigma'(y)$ and $\tau_2(x,y)=\sqrt{1-\rho^2}\sigma'(y)$ for a given $\sigma':\R\rightarrow \R$. In addition, instead of the VIX' constraint, which is a constraint on the convex order, the VIX \eqref{vix initial} was considered. Our Theorem \ref{main5} gives a proof of this duality formula.
\vspace{1mm}

\noindent
(iii) In Proposition \ref{martingality} we show that $D'_{\mathbb{P}_0}$ admits a minimizer $\mathbb{P}^*$. If the solutions of \eqref{HJvix'} and \eqref{HJvix} satisfy appropriate regularity assumptions, then we may write a formula for $\frac{d\mathbb{P}^*}{d\mathbb{P}_0}$ as in Theorem \ref{main4}.\\
(iv) Our assumptions on the coefficients of \eqref{SVM} allow us to apply our results from Theorem \ref{main4} and \ref{main5} to a plethora of well-known stochastic volatility models that appear in the literature. Two examples are the following, where we write $\text{Id}(y):=y$ for the identity function.
\begin{enumerate}
    \item SABR model: $\sigma(x,y)=xe^{\text{Id}(y)}$, $\tau_1=0$, $\tau_2=$constant, and $b(x,y)=-\frac{\tau^2_2}{2}$.
    \item Heston model: $\sigma(x,y)=x\sqrt{\text{Id}(y)}$, $b(x,y)=\kappa(\theta-y)$, $\tau_1(x,y)=\rho\xi\sqrt{\text{Id}(y)}$ and $\tau_2(x,y)=\xi\sqrt{1-\rho^2}\sqrt{\text{Id}(y)}$, where $\rho\in (-1,1)$ and $\kappa,\;\theta,\; \xi$ are given constants.
\end{enumerate}
However, due to our assumption of $\tilde{\sigma}$ being bounded or Lipschitz and $\tau_1,\tau_2$ being Lipschitz, we have to consider truncations of the identity function; depending on the case. In this case, our results become more appropriate for short maturities $T_1,\;T_2$.
\vspace{1mm}

\noindent
(v) The continuity of function $\Phi$ when $\tau_2$ is constant is proved in Appendix B (Lemma \ref{continuity estimates lemma}). Note that it is possible to show that $\Phi$ is lower semi-continuous (first part of step 2 in the Lemma \ref{continuity estimates lemma}) even if $\tau_2$ is non-constant. The assumption that $\tau_2$ is constant is essential in the proof of the upper semi-continuity and in particular, in the proof of estimate \eqref{"Lip"estimate}.
\vspace{1mm}

\noindent
(vi) Our VIX' constraint \eqref{VIX'} implies the following bound for VIX put option prices
\be\label{optionpricebound}
\mathbb{E}^{\mathbb{P}}[(K-\text{VIX}_{\mathbb{P}})_+]\leq \int\left(K-100\sqrt{\frac{2}{T_2-T_1}x}\right)_+d\mu_3(x),
\ee
where $K\in [0,+\infty)$ is a given strike price (see Proposition \ref{VIXoptionprice}). An interesting question for future investigation is whether our proof technique can be adapted to incorporate the VIX constraint \eqref{vix initial} in \eqref{minimization}, rather than only the weaker VIX' constraint \eqref{VIX'}, making \eqref{optionpricebound} an equality. The main technical difficulty is the application of a minimax theorem in \eqref{equation1001} in the proof of Theorem \ref{main5}.

\end{remark}


\section{Notation, Assumptions and Preliminary results}

\subsection{Notation}
Throughout the note, $d\in\mathbb{N}$ is a given dimension,  $T,T_1,T_2>0$ with $T_1<T_2$. $\R_{\ge 0}= [0,+\infty)$.
$\gamma^{d}$ is the $d-$dimensional standard normal distribution in $\R^d$. For $k\geq 0$, $C^k_b(\R^d)$ is the space of all $k$-continuously differentiable functions with bounded derivatives. We write $g\in \text{Lip}$ if $g$ is Lipschitz and bounded function. For a function $u=u(t,x):[0,T]\times \R\rightarrow \R$ depending on time $t$ and space $x$, we use the symbol $Du$ for the partial derivative $D_xu$.

\noindent
We use the symbol $\mathcal{P}(\R)=\mathcal{P}_2(\R)$ for all probability measures with finite second moments. We denote by ${\bf d}_2(\mu,\nu)=\inf_{X\sim \mu, Y\sim\nu}\left(\mathbb{E}[|X-Y|^2]\right)^{1/2}$ the Wasserstein 2-distance between two measures $\mu,\nu\in \mathcal{P}(\R)$. If $f:\R\rightarrow \R$ is a function, then $f_\#\mu\in \mathcal{P}(\R)$ is the measure such that $f_\#\mu(A)=\mu(f^{-1}(A))$ for every measurable $A\subset \R$. \\
We say that two probability measures $\mu,\nu\in\mathcal{P}(\R)$ are in convex order and we write $\mu \leq_c \nu$ if $\int h(x)d\mu(x)\leq \int h(x)d\nu(x)$ for every convex  function $h:\R\rightarrow \R$.\footnote{It is straightforward to show that if a flow of probability measures $(m_t)_{t\in[0,T]}$ satisfies $\partial_t m=\frac{1}{2}\partial_{xx}(m b)$, for some $b: [0,T]\times \R\rightarrow \R_{\ge 0}$ then $t\mapsto m_t$ is increasing with respect to the convex order.} We also write $\mu\leq_{c,l}\nu$ if $\int h(x)d\mu(x)\leq \int h(x)d\nu(x)$ for any convex and lower bounded $h:[0,+\infty)\rightarrow \R$.\\
The relative entropy between two measures $\mu,\nu$ is defined as follows.
$$H(\mu|\nu)=\begin{cases}
    \int \log\frac{d\mu}{d\nu}d\mu, & \text{if } \mu\ll \nu,\\
    +\infty, & \text{ otherwise.}
\end{cases}$$
When $X$ is a random variable over a probability space $(\Omega,\mathbb{P})$ with distribution $\mu$ we write $X\stackrel{\mathbb{P}}{\sim} \mu$ or $\mathcal{L}^{\mathbb{P}}$ or simply $X\sim \mu$, $\mathcal{L}(X)=\mu$, respectively, when there is no confusion about $\mathbb{P}$.\\
If $(\Omega,\mathcal{F},\mathbb{P})$ is a probability space, we sometimes abuse the notation and we write $V\in \mathcal{F}$ if $V$ is an $\mathcal{F}$-measurable function. We also use the notation $\mathbb{E}^{\mathbb{P}}$ for the expectation with respect to the probability measure $\mathbb{P}$. $\mathcal{P}$ is the space of probability measures over $(\Omega,\mathcal{F},\mathbb{F})$.\\
$\ind_A$ is the indicator function of a given set $A$.

\subsection{Assumptions}
We make the following assumptions.
\vspace{1.5mm}

\noindent
(A1) $L:\R\rightarrow \R$ is convex and $p-$coercive for some $p\in [2,+\infty)$.
That is, there exists a constant $C$ such that $L(a)\geq C|a|^p$ for every $a\in \R$.
\vspace{2mm}

\noindent
(A2) $H: \R\rightarrow \R$ with $H(a)=\sup_{b\ge 0}\{-ab-L(b)\}$ is $C^1$. 
\vspace{2mm}

\noindent
(A3) The function $\sigma:\R^2\rightarrow \R$ has the form $\sigma(x,y)=x\tilde{\sigma}(y)$ for some Lipschitz function $\tilde{\sigma}:\R\rightarrow \R$. The functions $b,\tau_1,\tau_2: \R^2\rightarrow \R$ are $C^2$ and Lipschitz functions. $\tau_2$ satisfies 
$$|Db(x,y)|+|D\tau_2(x,y)|\leq \lambda_0 |\tau_2(x,y)|,\text{ for all }x,y\in \R,$$
for some constant $\lambda_0>0$. Finally, 
$$|\tau_1(x,y)|+|\tau_2(x,y)|+|b(x,y)|\le C(1+|y|),\text{ for all }x,y\in \R.$$

\noindent
(A4) The functions $b:\R^2\rightarrow \R$ and $\tilde{\sigma}:\R\rightarrow \R$ are as in assumption (A3) and satisfy: (i) $\tilde{\sigma}$ is bounded and (ii) $(x,y)\mapsto b(e^x,y)$ is Lipschitz.

\begin{remark}[Technical points]
(i) Assumption (A1) allows us to prove the duality formula \eqref{duality} by providing the compactness needed in order to apply the Von-Neumann theorem (Theorem \ref{Von}).\\
(ii) It is clear that $H$ is convex and degenerate elliptic; that is its derivative $H'(a)\leq 0$ for any $a\in \R$. This allows us to use the theory of viscosity solution for the equation \eqref{HJmain} (see section \ref{HJequation}). Assumption (A2) is necessary for the statement of Corollary \ref{cor1}.
\vspace{1mm}

\noindent
(iii) Assumption (A3) establishes the well-posedness of the stochastic system \eqref{SVM}. Moreover, the $C^2$ regularity assumption gives us the regularity of the solution of \eqref{HJS} (Corollary \ref{reg1}).
\vspace{1mm}

\noindent
(iv) Assumption (A4) is used to establish uniqueness of viscosity solutions with $\log^2$ growth of \eqref{HJvix'} via \cite[Theorem~2.1]{da2006uniqueness}. The boundedness of $\tilde{\sigma}$ in (A4) is further required to obtain certain integrability properties for $(X_t)_{t \geq t_0}$, which are essential for applying the Von-Neumann minimax theorem in the proof of \eqref{dual101}. A main class of examples of functions $b$ satisfying (A3) and (A4) is given when $b(x,y)$ is independent of $x$. 

\end{remark}

\subsection{Preliminary result}
Let $0\le t_0\le T$. In the proofs of our main results, we will be using two compactness results related to Fokker-Planck equations. The first one is associated with the equation
\be \label{FP}
\partial_t m=\frac{1}{2}\partial_{xx}(b_0m),
\ee
in dimension 1, where $b_0:[t_0,T]\times \R\rightarrow \R_{\ge 0}$ such that $b_0\in L^2(dm(s)\otimes ds)$. The second is for the equation of the form\footnote{Note that, at least formally, \eqref{FP} describes the distribution of the stochastic process satisfying the SDE $dX_t=\sqrt{b_0(X_t)}dW^1_t$, while for \eqref{FPdrift} the SDE is $dX_t=(\tilde{\tau}(X_t)+\tau(X_t)\alpha_k)dt+\sqrt{2}\sigma_0(X_t)dW^n_t$, where $W^1,W^n$ are 1-dimensional and $n$-dimensional Brownian motions, respectively.}
\be\label{FPdrift}
\partial_tm_k-\sum_{i,j=1}^n\partial_{ij}\left( (\sigma_0\sigma_0^\top)_{ij} D^2m_k\right)+\text{div}\left((\tilde{\tau}+\tau \alpha_k)m_k\right)=0,
\ee
for some given $\tilde{\tau}:\R^n\rightarrow \R^n$, $\sigma_0,\tau:\R^n\rightarrow \R^{n\times n}$, a dimension $n\in\mathbb{N}$ and measurable functions $(\alpha_k)_{k\in\mathbb{N}}$ with values in $\R^n$. We state this proposition.

\begin{proposition}\label{comp2}
    Let $n\in\mathbb{N}$ be a dimension. Suppose that $\tilde{\tau}:\R^n\rightarrow\R^n$ and $\sigma_0,\tau: \R^n\rightarrow \R^{n\times n}$ are Lipschitz, $\tau$ is bounded, and that for all $k\ge 1$, $(m_k,\alpha_k)$ solves \eqref{FPdrift} (in the sense of distributions) starting from the probability measure $m_0\in \mathcal{P}_2(\R^n)$ at time $t_0$. We also assume that $(m_k,\alpha_k)$ satisfies the uniform energy estimate
$$\int_{t_0}^T\int_{\R^d}|\alpha_k(t,x)|^2dm_k(t)(x)dt\leq C,$$ 
for all $k\in\mathbb{N}$ and for some constant $C$ independent of $k$. Then, for any $\delta\in (0,1)$, up to taking a subsequence, $(m_k,\alpha_km_k)$ converges in $\mathcal{C}^{\frac{1-\delta}{2}}\left( [0,T];\mathcal{P}_{2-\delta}(\mathbb{R}^d)\right)\times\mathcal{M}([0,T]\times\mathbb{R}^d,\R^d)$ to some $(m,w)$. The curve $m$ is in $\mathcal{C}^{1/2}\left( [0,T],\mathcal{P}_2(\mathbb{R}^d)\right)$, $w$ is absolutely continuous with respect to $dm(t)\otimes dt$, 
\be \nonumber
\int_{t_0}^{T}\int_{\mathbb{R}^d} \bigg| \frac{dw}{dm(t)\otimes dt}(t,x)\bigg|^2dm(t)(x)dt\leq \liminf_{k\rightarrow +\infty}\int_{t_0}^{T}\int_{\mathbb{R}^d}|\alpha_k(t,x)|^2dm_k(t)(x)dt
\ee
and $(m,\frac{dw}{dm(t)\otimes dt})$ solves \eqref{FPdrift} starting from $m_0$.
\end{proposition}

\begin{proof}
Let $k\in \mathbb{N}$ and consider the process $(X_t^k)_{t\in [t_0,T]}$ satisfying
$$\begin{cases}
    dX_t^k= (\tilde{\tau}(X_t^k)+\tau(X_t^k)\alpha_k(t,X_t^k))dt+\sqrt{2}\sigma_0(X_t^k)dW^n_t,\;\;\;T\ge t\ge t_0,\\
    X_{t_0}^k\sim m_0,
\end{cases}$$
where $W^n$ is an $n$-dimensional Brownian motion. We have $X_t^k\sim m_k(t)$, with $m_k$ satisfying \eqref{FPdrift}. Due to our assumptions on $\tilde{\tau},\sigma_0,\tau$ and  standard Gr\"onwall arguments, we obtain the $L^2$ estimate
$$\mathbb{E}[|X_t^k|^2]\leq C\left(\int |x|^2dm_0(x)+1\right),$$
where $C$ is independent of $k,t$ and depends only on $T$ and $\tilde{\tau},\sigma_0,\tau$. By using this $L^2$ estimate, Cauchy-Schwartz and the given energy estimate, we can now write for all $t,s\in [t_0,T]$ with $s<t$
\begin{align*}
    \mathbb{E}[|X_t^k-X_s^k|^2]&\leq 3\mathbb{E}\left[ \left|\int_s^t\tilde{\tau}(X_u^k)du \right|^2\right]+3\mathbb{E}\left[ \left|\int_s^t\tau(X_u^k)\alpha_k(u,X^k_u)du \right|^2\right]+6\mathbb{E}\left[ \int_s^t\sigma_0^2(X^k_u)du  \right]\\
    & \leq C(t-s)\left( \mathbb{E}\left[\int_s^t(|X_u|^2+1)du \right]+\|\tau\|_{\infty}^2\mathbb{E}\left[\int_s^t\alpha_k^2du \right] \right)+C\mathbb{E}\left[\int_s^t(|X_u|^2+1)du \right]\\
    &\leq C(t-s),
\end{align*}
where $C$ changes from line to line. The proof now can be finished as in \cite[Proposition 1.2]{daudin2023optimal}. We omit these details.
\end{proof}

\begin{proposition}\label{comp}
Suppose that (A1) holds. Assume that, for all $k\geq 1$, $(m_k,b_k)$ solves the Fokker-Planck equation \eqref{FP} (in the sense of distributions) starting from $m_0$ at time $t_0$ and satisfies the uniform energy estimate
\be \label{est1}
\int_{t_0}^T\int_{\R} L(b_k(t,x))dm_k(t)(x)dt\leq C,
\ee
for some constant $C>0$ independent of $k$. Then, there exists $\delta\in (0,1)$ such that, up to taking a subsequence, $(m_k,b_km_k)$ converges in $\mathcal{C}^{\delta}\left( [t_0,T];\mathcal{P}_{2}(\R)\right)\times\mathcal{M}([t_0,T]\times\R,\R)$ to some $(m,w)$. The curve $m$ is in $\mathcal{C}^{\delta}\left( [t_0,T],\mathcal{P}_2(\R)\right)$, $w$ is absolutely continuous with respect to $dm(t)\otimes dt$, it holds that
\be \label{limitineq}
\int_{t_0}^{T}\int_{\R} L\bigg( \frac{dw}{dm(t)\otimes dt}(t,x)\bigg)dm(t)(x)dt\leq \liminf_{k\rightarrow +\infty}\int_{t_0}^{T}\int_{\R}L\left(b_k(t,x) \right)dm_k(t)(x)dt
\ee
and $(m,\frac{dw}{dm(t)\otimes dt})$ solves \eqref{FP} starting from $m_0$.
\end{proposition}
\begin{proof}
 Note that since $b_k$ is non-negative, we may write $b_k=\sigma_k^2$, where $\sigma_k$ is the square root of $b_k$. Then, \eqref{FP} describes the law of the process $dX_t=\sigma_k(t,X_t)dB_t$. For any $s,t\in [t_0,T]$, by It\^o's formula we have
    \begin{align*}
        {\bf d}_2(m_s,m_t)&\leq \mathbb{E}[ |X_s-X_t|^2]^{1/2}=\bigg(\int_{s}^t\mathbb{E}[\sigma_k^2(\tau,X_{\tau}) ]d\tau\bigg)^{1/2}\\ &\leq (s-t)^{\delta}\bigg(\int_{t_0}^T\int L(b_k(\tau,x))dm_k(\tau)(x)d\tau\bigg)^{1-\delta}
        \leq C(s-t)^{\delta},
    \end{align*}
    for some constant $0<\delta<1$, because of the assumptions on $L$ and H\"older's inequality. Also, due to \eqref{est1} and the coercivity of $L$, the sequence $b_km_k$ has bounded variation. By Banach-Alaoglu and Arzela-Ascoli we deduce the convergence of the pair $(m_k,b_km_k)$ to some $(m,w)$ as claimed.

\noindent
    By Theorem 2.34 of \cite{ambrosio2000functions}, we discover that $w$ is absolutely continuous with respect to $dm(\tau)\otimes d\tau$ and that \eqref{est1} holds. Let $b$ be the Radon-Nikodym derivative of $w$ with respect to $dm(\tau)\otimes d\tau$. It is straightforward to show that $(m,wm)$ is a solution to \eqref{FP} starting from $m_0$.
\end{proof}

\noindent
We will also need the following lemma.

\begin{lemma}\label{condint}
    Let $0<T_1<T_2$, $(\Omega,\mathcal{F},\mathbb{F}=(\mathcal{F}_t)_{t\ge 0}, \mathbb{P}_0)$ be a filtered probability space and $(f_n)_{n\in\mathbb{N}}: \Omega\rightarrow [0,+\infty)$ a sequence of probability densities that converges weakly in $L^1$ to some probability density $f$. Assume further that $(\mathbb{P}_n)_{n\in\mathbb{N}}$ are probability measures over $(\Omega,\mathcal{F})$ such that $d\mathbb{P}_n=f_nd\mathbb{P}$ and $V$ is a random variable in $(\Omega,\mathcal{F})$.
    The following statements hold: (i) $(\mathbb{P}_n)_{n\in\mathbb{N}}$ converges weakly to the probability measure $\mathbb{P}$, which is such that $d\mathbb{P}=fd\mathbb{P}_0$, (ii) if $V$ is lower bounded, then $\liminf_n \mathbb{E}^{\mathbb{P}_0}[Vf_n]\ge \mathbb{E}^{\mathbb{P}_0}[Vf]$, and (iii) if $V$ is bounded or uniformly integrable with respect to $(\mathbb{P}_n)_{n\in\mathbb{N}}$, then $\mathbb{E}^{\mathbb{P}_0}[Vf_n|\mathcal{F}_{T_1}]\rightharpoonup \mathbb{E}^{\mathbb{P}_0}[Vf|\mathcal{F}_{T_1}]$ weakly in $L^1$.
\end{lemma}

\begin{proof}
    For (i), we have for any $g$ bounded and continuous
    $$\int_\Omega gd\mathbb{P}_n=\int_\Omega g f_nd\mathbb{P}_0\xrightarrow{n\rightarrow\infty}\int_\Omega gfd\mathbb{P}_0=\int_\Omega gd\mathbb{P}.$$
    For (ii), let $N\in \mathbb{N}$ and $K\in \R$ such that $V\ge K$ almost surely. We observe that
    \begin{align*}
        \mathbb{E}^{\mathbb{P}_0}[Vf_n]=\mathbb{E}^{\mathbb{P}_0}[(V-K)f_n]+K\geq \mathbb{E}^{\mathbb{P}_0}[\ind_{\{V-K\le N\}}(V-K)f_n]+K.
    \end{align*}
    Since $\ind_{\{V-K\le N\}}(V-K)$ is bounded, by sending $n\rightarrow +\infty$ we get
    \begin{align*}
        \liminf_{n\rightarrow +\infty}\mathbb{E}^{\mathbb{P}_0}[Vf_n]\ge \mathbb{E}^{\mathbb{P}_0}[\ind_{\{V-K\le N\}}(V-K)f]+K.
    \end{align*}
    The result follows by letting $N\rightarrow +\infty$.

\noindent
For (iii), we start with the case where $V$ is bounded. Let $Z$ be a bounded $\mathcal{F}_{T_1}$-measurable random variable. Then, we have
\begin{align*}
    \int Z \mathbb{E}^{\mathbb{P}_0}[Vf_n|\mathcal{F}_{T_1}]d\mathbb{P}_0=\int ZVf_nd\mathbb{P}_0\xrightarrow{n\rightarrow +\infty}\int ZVfd\mathbb{P}_0=\int Z\mathbb{E}^{\mathbb{P}_0}[Vf|\mathcal{F}_{T_1}]d\mathbb{P}_0.
\end{align*} 
Now we assume that $V$ is uniformly integrable with respect to $(\mathbb{P}_n)_{n\in \mathbb{N}}$. Let $\e>0$. By the uniform integrability of $V$, there exists $R>0$ such that $\int_{\{|V|>R\}} |V|d\mathbb{P}_n<\e$ for any $n\in\mathbb{N}$. We have for $Z$ bounded $\mathcal{F}_{T_1}$-measurable
    $$\int_\Omega VZf_nd\mathbb{P}_0\le \int_{\{|V|>R\}}VZd\mathbb{P}_n+ \int_{\{|V|\leq R\}}VZf_nd\mathbb{P}_0<\e\|Z\|_{\infty}+ \int_{\{|V|\leq R\}}VZf_nd\mathbb{P}_0,$$
    therefore, using the result for bounded $V's$, $\limsup_n\int_\Omega VZf_nd\mathbb{P}_0\le \e\|Z\|_{\infty}+\int_{\{|V|\le R\}}VZfd\mathbb{P}_0$. We now send $R\rightarrow +\infty$ to derive
    $$\limsup_n\int_\Omega VZf_nd\mathbb{P}_0\le \e\|Z\|_{\infty}+\int VZfd\mathbb{P}_0\implies \limsup_n\int_\Omega VZf_nd\mathbb{P}_0\le \int VZfd\mathbb{P}_0.$$
    With similar arguments, we can prove that
    $$\liminf_n\int_\Omega VZf_nd\mathbb{P}_0\ge \int VZfd\mathbb{P}_0.$$
    The result now follows as in the case where $V$ is bounded.
\end{proof}

\section{Martingale Optimal Transport, Proof of Theorem \ref{main2}}

\noindent
We start by noting the following non-probabilistic interpretation of \eqref{mot}.

\begin{proposition}\label{nonprob}
    Under the assumptions of Theorem \ref{main2}, we have that the minimization problem $V(T_0,\mu_0,\mu_1,\mu_2)$ is admissible. Furthermore, the following equality is true
\be \label{opt1}    
    \begin{split} 
V(T_0,\mu_0,\mu_1,\mu_2)=\inf_{b,m} \bigg\{ \int_{T_0}^{T_2}&\int L(b(s,x))dm_s(x)ds\\
&\bigg|\partial_t m=\frac{1}{2}\partial_{xx}(m b),\; m(T_0)=\mu_0,\;  m(T_1)=\mu_1, \; m(T_2)=\mu_2\bigg\},
\end{split}
\ee
where the infimum is taken over all $b:[T_0,T_2]\times \R\rightarrow \R_{\ge 0}$ with $b\in L^2(dm_s\otimes ds)$ and we require the equation for $m$ in \eqref{opt1} to hold in the weak sense.  
\end{proposition}

\begin{proof}
    The admissibility follows from Theorem \ref{main3} below. The proof of \eqref{opt1} uses a mimicking theorem and can be found in \cite[Lemma 3.1, Proposition 3.4]{guo2022calibration}.
\end{proof}
\vspace{3mm}

\begin{remark}
To motivate the introduction of the Dirac delta in \eqref{HJmain}, we consider that at a fixed time $t\in [0,T_1]$ the distribution is $\mu_0$ and we observe that, at least formally, we have
\begin{align*}
V(t,\mu_0,\mu_1,&\mu_2)=\inf_{b,m}\sup_{u_1,u_2}\bigg\{ \int_t^{T_2}\int L(b(s,x))dm_s(x)ds +\int u_1dm_{T_1}\\
&\hspace{6cm}-\int u_1d\mu_1 +\int u_2 dm_{T_2}-\int u_2 d\mu_2\bigg\}\\
    &\hspace{-0.5cm}=\sup_{u_1,u_2}\inf_{b,m}\bigg\{ \int_t^{T_2}\int L(b(s,x))dm_s(x)ds +\int u_1dm_{T_1}+\int u_2 dm_{T_2}-\int u_1d\mu_1-\int u_2 d\mu_2\bigg\}\\
    &\hspace{-1.2cm}=\sup_{u_1,u_2}\bigg\{ \inf_{b,m}\bigg\{ \int_t^{T_2}\int L(b(s,x))dm_s(x)ds +\int u_1dm_{T_1}+\int u_2 dm_{T_2}\bigg\} -\int u_1d\mu_1-\int u_2 d\mu_2\bigg\}.
\end{align*}
Studying the infimum that is inside the supremum is standard in stochastic optimal control theory (e.g \cite[Chapter 3]{cardaliaguet2019master}), when instead of the term $\int u_1dm_{T_1}$ we have $\int_t^{T_2}F(s,m(s))ds$ for some running cost function $F:[0,T_2]\times \mathcal{P}(\R)\rightarrow \R$. In our case, formally, $F(s,m)=\int \delta_{T_1}(s)u_1(x)dm(x)$ with $\frac{\delta F}{\delta m}(s,m,x)=\delta_{T_1}(s)u_1(x)$ and hence the optimizer is expected to be associated with the Hamilton-Jacobi equation \eqref{HJmain}. In particular, the optimal control is $b(s,x)=-H'\left(\frac{D^2 u(s,x)}{2}\right),\; s\in [t,T_2]$, where $u$ solves \eqref{HJmain}.
\end{remark}
\vspace{3mm}

\noindent 
To study \eqref{opt1}, we start with its simplified version; namely when there is no intermediate constraint $m(T_1)=\mu_1$. We note that if $T_0\in (T_1,T_2]$, then \eqref{opt1} becomes independent of $\mu_1$, therefore $V(T_0,\mu,\mu_1,\mu_2)$ is identical to
\be\label{opt2}
V_2(T_0,\mu_0,\mu_2)=\inf_{b,m}\bigg\{  \int_{T_0}^{T_2}\int L(b(s,x))dm_s(x)ds\bigg| \partial_t m=\frac{1}{2}\partial_{xx}(m b),\; m(T_0)=\mu_0, \; m(T_2)=\mu_2\bigg\}.
\ee

\noindent
If $T_0\in[0,T_1]$, for any $b_1,b_2$ admissible functions satisfying
\be\label{5.1}
\begin{cases}
\partial_t m_1=\frac{1}{2}\partial_{xx}(m_1b_1),\\
m_1(T_0)=\mu_0,\;\;m_1(T_1)=\mu_1,
\end{cases}
\text{ and }
\begin{cases}
\partial_t m_2=\frac{1}{2}\partial_{xx}(m_2b_2),\\
m_2(T_1)=\mu_1,\;\;m_2(T_2)=\mu_2,
\end{cases}
\ee
we write $b(s,x)=1_{[T_0,T_1]}(s)b_1(s,x)+1_{(T_1,T_2]}(s)b_2(s,x)$ and $m(s)=1_{[T_0,T_1]}(s)m_1(s)+1_{(T_1,T_2]}(s)m_2(s)$. The pair $(m,b)$, then satisfies in the weak sense
\be\nonumber
\begin{cases}
\partial_tm=\frac{1}{2} \partial_{xx}(mb),\\
m(T_0)=\mu_0,\; m(T_1)=\mu_1,\; m(T_2)=\mu_2,
\end{cases}
\ee
and hence it is admissible for \eqref{opt1}. Therefore, depending on the value of $T_0$, we can study \eqref{opt1} as
\begin{align*}
V_1(T_0,\mu_0,\mu_1,\mu_2)=\inf_{b,m}\bigg\{ \int_{T_0}^{T_1}\int L(b(s,x))dm_s(x)ds+V_2(T_1,\mu_1,\mu_2)\bigg| &\partial_tm=\frac{1}{2}\partial_{xx}(mb),\\
&\; m(T_0)=\mu_0, m(T_1)=\mu_1\bigg\},
\end{align*}
if $T_0\in [0,T_1]$ and $V(T_0,\mu,\mu_1,\mu_2)=V_2(T_0,\mu,\mu_2)$, where $V_2$ is given in \eqref{opt2}, if $T_0\in [T_1,T_2]$. To prove Theorem \ref{main2}, we will start by proving its following simplified version.

\begin{theorem}\label{main3}
Assume (A1) and (A2). Let $\mu,\nu\in\mathcal{P}(\R)$ be two probability measures in convex order $\mu\leq_c\nu$, $T>0$, $T_0\in [0,T]$ and 
\be \label{opt4}
U(T_0,\mu,\nu):=\inf_{b,m}\bigg\{ \int_{T_0}^{T}\int L(b(s,x))dm_s(x)ds\bigg| \partial_tm=\frac{1}{2}\partial_{xx}(mb),\; m(T_0)=\mu, m(T)=\nu\bigg\}.
\ee
Then, $U(T_0,\mu,\nu)$ is admissible and, whenever it is finite,
\be\label{duality2}
U(T_0,\mu,\nu)=\sup_{g\in \text{Lip}}\bigg\{ -\int g(x)d\nu(x)+ \int u(T_0,x)d\mu (x))      \bigg\},
\ee
where the supremum is over all bounded and Lipschitz functions $g$ and where $u$ is a viscosity solution of the Hamilton-Jacobi-Bellman equation
\be \label{HJBmfp}
\begin{cases}
    -\partial_t u+H\left( \frac{D^2u}{2}\right)=0,& \text{ in }[T_0,T)\times \R,\\
    u(T)=g, & \text{ in }\R.
\end{cases}
\ee
\end{theorem}
\vspace{2mm}

\noindent
Before presenting the proof of Theorem \ref{main3}, we will show the following Proposition which is related to the Hamilton-Jacobi-Bellman equation \eqref{HJBmfp}.

\begin{proposition} \label{proof of HJBmfp}
    Assume that (A1), (A2) hold. Let $g$ be a Lipschitz function and $\mu\in\mathcal{P}(\R)$. Then, the function $U_g:[0,T]\times \mathcal{P}(\R)\rightarrow \R$ with
    $$U_g(t,\mu)=\inf_{b,m}\bigg\{ \int_t^T\int L\left(b(x,s)\right) dm_s(x)ds+\int g(x)dm_T(x)\bigg\},$$
    where the infimum is taken over all $b:[t,T]\times \R\rightarrow \R_{\ge 0}$ and $m\in \mathcal{P}(\R)$ such that $\partial_sm=\frac{1}{2}\partial_{xx}(bm)$ in the weak sense and $m(t)=\mu$, takes the form $\int u(t,x) d\mu(x)$, where $u$ is a viscosity solution of \eqref{HJBmfp}.
\end{proposition}

\begin{proof}
Fix $(t,\mu) \in [0,T]\times\mathcal{P}(\R)$. We notice that $U_g(t,\mu)$ is a mean field control problem with controls over the diffusion. By the linearity with respect to $m$ of the functional inside the infimum, we have that
$$U_g(t,\mu)=\int_{\R^d} \widetilde{U}_g(t,x) \mu(dx),$$
where $\widetilde{U}_g(t,x)$ solves (in the classical sense) the stochastic optimal control problem $\widetilde{U}_g(t,x)=U_g(t,\delta_x)$. However, by standard stochastic optimal control theory, $U_g(t,\delta_x)=u(t,x)$, where $u$ solves the Hamilton-Jacobi equation
\be\label{HJ for prop}
\begin{cases}
-\partial_s u+H\left(\frac{1}{2}D^2u\right)=0, &\text{ in } [t,T)\times \R,\\
 u(T,x)=g(x), &\text{ in }\R,
\end{cases}
\ee
in the viscosity sense.
\end{proof}

\vspace{2mm}

\noindent
We are now ready to prove Theorem \ref{main3}.

\begin{proof}
\textit{Step 1.} We will start by showing that the minimization problem \eqref{opt4} is admissible. The proof follows an idea from \cite{backhoff2017martingale}.
\vspace{1mm}

\noindent
Since $\mu\leq_{c}\nu$, by Strassen's theorem \cite{strassen1965existence}, there exists a filtered probability space $(\Omega,\mathcal{F},(\mathcal{F}_s)_{s\geq 0},\mathbb{P})$ and a discrete martingale $(X,Y)$ adapted to the filtration $(\mathcal{F}_0,\mathcal{F}_1)$ such that $\mathcal{L}^{\mathbb{P}}(X)= \mu$ and $\mathcal{L}^{\mathbb{P}}(Y)= \nu$. We consider $\pi$ to be the law of $(X,Y)$ and we write $\pi(dx,dy)=\pi_x(dy)\otimes \mu(dx)$. For any $x\in \R^d$, by Brenier's theorem \cite[Theorem 2.12]{villani2021topics}, there exists a convex function $\phi^{x}$ such that $(D\phi^x)_{\#}\gamma=\pi_x$. Now, let $B$ be a standard Brownian motion adapted to the filtration $(\mathcal{F}_s)_{s\geq 0}$ and
\be\nonumber
M_s:=\mathbb{E}\left[ D \phi^X (B_1)|\mathcal{F}_s \right].
\ee
It is obvious that $M_t$ is a continuous martingale and for any bounded continuous function $h$
\begin{align*}
    M_0&=\mathbb{E}[D\phi^X(B_1)|\mathcal{F}_0]=\mathbb{E}[D\phi^X(B_1)|X]=\mathbb{E}[\int D \phi^X(y)\gamma(dy)|X]=\mathbb{E}[\int z \pi_X(dz)|X]=X,\\
    \mathbb{E}[h(M_1)]&=\mathbb{E}[h(D\phi^X(B_1))|\mathcal{F}_1]=\mathbb{E}[h(D\phi^X(B_1))]=\int h(y)(D\phi^x)_{\#}\gamma(dy)\mu(dx)=\int h(y)\pi(dx,dy)\\  
    &=\int h(y)\nu(dy),
\end{align*}
hence $\mathcal{L}^{\mathbb{P}}(M_0)= \mu$ and $\mathcal{L}^{\mathbb{P}}(M_1)=\nu$. By the martingale representation theorem there exists a process $\sigma_s$ with $\int_{t_0}^t|\sigma_s|^2ds<+\infty$ almost surely for each $t\geq 0$, such that $dM_s=\sigma_s dB_s$. However, by the mimicking theorem \cite[Corollary 3.7]{brunick2013mimicking}, there exists a measurable function $\tilde{\sigma}(s,x)=\mathbb{E}[\sigma_s|M_s=x]$ such that the process solving the stochastic differential equation $d\tilde{M}_s=\tilde{\sigma}(s,\tilde{M}_s)dB_s$ on a possibly different probability space satisfies $\mathcal{L}^{\mathbb{P}}(\tilde{M}_s)=\mathcal{L}^{\mathbb{P}}(M_s)$ for every $s\geq 0$. Moreover, the curve  $t\mapsto m_t=\mathcal{L}^{\mathbb{P}}(\tilde{M}_s)$ satisfies \eqref{FP} in the weak sense with $b(s,x)=\tilde{\sigma}\tilde{\sigma}^{\top}(s,x)$, initial condition $m_0=\mu$ and terminal condition $\mu_1=\nu$. Therefore, the pair $s\mapsto \tilde{m}_s=m_{\frac{s-t_0}{T-t_0}}$ and $\tilde{\sigma}$ is an admissible candidate for the minimization problem \eqref{opt4}.
\vspace{2mm}

\noindent
\textit{Step 2.} We will now show \eqref{duality2}. For every admissible $b$, let $W$ be the measure having density $b$ with respect to $dm_t\otimes dt$. We consider the functional
    $$\mathcal{L}(g,(W,m))= \int_{T_0}^T\int L\left(\frac{dW}{dm_s\otimes ds}\right)dm_sds+\int g(x)dm_T(x)-\int g(x)d\nu(x),$$
    where $g$ is a Lipschitz function.
    We will show that $\mathcal{L}$ satisfies the conditions of the Von-Neumann theorem \ref{Von}. Indeed, $\mathcal{L}(g,(W,m))$ is concave with respect to $g$ and convex with respect to $(W,m)$. We now fix a $g_*$ and let $C_*$ such that $C_*>\sup_g\inf_{W,m}\mathcal{L}(g,(W,m))$. Then, if $(W_n,m_n)$ is a sequence such that $\mathcal{L}(g_*,(W^n,m^n))\leq C_*$, we discover that 
    $$\int_{T_0}^T\int L\left(\frac{dW^n}{dm^n_s\otimes ds}\right) dm^n_s(x)ds\leq \tilde{C},$$
    for some other constant $\tilde{C}$. Hence, by Proposition \ref{comp}, $(W^n,m^n)$ converges up to a subsequence to an admissible $(W,m)$ and, by passing to the limit, such that $\mathcal{L}(g_*, W,m)\leq C_*$. Finally, due to our argument from Proposition \ref{comp} once again, $(W,m)\mapsto \mathcal{L}(g, W,m)$ is lower-semicontinuous for every $g$, thus Theorem \ref{Von} is applicable. We have
    \begin{align*}
        U(T_0,\mu,\nu)&=\inf_{W,m}\sup_{g\in \text{Lip}}\mathcal{L}(g,W,m)\\
        &= \sup_{g\in\text{Lip}}\inf_{W,m}\bigg\{ \int_{T_0}^T\int L\left( \frac{dW}{dm_s\otimes ds}\right)dm_sds+\int g(x)dm_T(x)-\int g(x)d\nu(x)\bigg\}\\
        &=\sup_{g\in\text{Lip}}\bigg\{ U_g(T_0,\mu)-\int g(x)d\nu(x)\bigg\},
    \end{align*}
    where 
    \be\nonumber
U_g(T_0,\mu)=\inf_{W,m}\bigg\{ \int_{T_0}^T\int L\left( \frac{dW}{dm_s\otimes ds}\right) dm_sds+\int g(x)dm_T(x)\bigg\},
    \ee
    and where $(W,m)$ satisfies $\partial_tm=\frac{1}{2}\partial_{xx}W$ with $m(T_0)=\mu$ in the sense of distributions. By Proposition \ref{proof of HJBmfp}, $U_g$ satisfies \eqref{HJBmfp} and the proof of \eqref{duality2} is complete.
\end{proof}

\noindent
We will now use Theorem \ref{main3} to prove Theorem \ref{main2}. In the calculations below, whenever there is no confusion, we are using the notation $u(t)$ for the function $x\mapsto u(t,x)$ and $\int u(t)dm:=\int u(t,x)dm(x)$ for a probability measure $m$.

\begin{proof}[Proof of Theorem \ref{main2}]
For $T_0\geq T_1$ the result follows from Theorem \ref{main3}, so we may assume that $T_0<T_1$. Arguing as in the proof of Theorem \ref{main3}, we can show that there exist $b_1, b_2$ such that \eqref{5.1} holds, therefore \eqref{opt1} is admissible. We can once again apply the Von-Neumann theorem to get
$$
V(T_0,\mu,\mu_1,\mu_2)=\sup_{u_1,u_2}\inf_{b,m}\bigg\{ \int_{T_0}^{T_2}\int L(b(s,x))dm_s(x)ds +\int u_1dm_{T_1}+\int u_2 dm_{T_2}-\int u_1d\mu_1-\int u_2 d\mu_2\bigg\}.
$$
By minimizing on $(T_1,T_2]$ first and then on $[T_0,T_1)$, we get by Theorem \ref{main3}
\begin{align*}
V(T_0,\mu,\mu_1,\mu_2)&=\sup_{u_1,u_2}\inf_{b,m}\bigg\{ \int_{T_0}^{T_1}\int L(b(s,x))dm_s(x)ds+\int (u_1+u(T_1^+))dm_{T_1}-\int u_1 d\mu_1 -\int u_2 d\mu_2 \bigg\}\\
&= \sup_{u_1,u_2}\bigg\{ \int u(T_0,x)d\mu(x)-\int u_1(x)d\mu_1(x)-\int u_2(x)d\mu_2(x)\bigg\},
\end{align*}
where $u$ on the first and second line is viscosity solution of
$$
\begin{cases}
    -\partial_t u +H\left( \frac{D^2 u}{2}\right)=0,\\
    u(T_2)=u_2
\end{cases}
\text{ and   }\;\;\;
\begin{cases}
    -\partial_t u +H\left( \frac{D^2 u}{2}\right)=0,\\
    u(T_1)=u_1+u(T_1^+),
\end{cases}
$$
respectively. We conclude that $u$ satisfies \eqref{HJmain} in the sense of Definition \ref{defvisc}. 
\vspace{1mm}

\noindent
Finally, we show that the supremum in \eqref{duality} can be also taken over classical super-solutions. We suppose that $T_0<T_1$. The case $T_0>T_1$ can be proved similarly. We let
\be\label{set}
D:=\sup_{v}\bigg\{ \int v(T_0)d\mu_0 -\int v(T_2)d\mu_2- \text{\bf{1}}_{[0,T_1]}(T_0)\int (v(T_1^{-})-v(T_1^{+}))d\mu_1\bigg\},
\ee
where the supremum is taken over all smooth $v$ such that $-\partial_tv +H\left(\frac{D^2v}{2}\right)\leq 0$ when $t\in [T_0,T_1)\cup (T_1,T_2]$. Note that if $u$ is a classical super-solution of \eqref{HJmain} and $m$ satisfies \eqref{FP}, then if $t_1<t_2<T_1$ or $T_1<t_1<t_2$ we have
\begin{align}
\int u(t_2,x)dm_{t_2}(x)-\int u(t_1,x)dm_{t_1}(x)&\geq \int_{t_1}^{t_2}\int H\left( \frac{D^2u}{2}\right)dm_s ds+\int_{t_1}^{t_2}\int  \frac{D^2u}{2}b(s,x)\;dm_s(x)ds\nonumber\\
&\geq -\int_{t_1}^{t_2}\int L(b(s,x))dm_s(x)ds.\label{test101}
\end{align}
We use \eqref{test101} for $t_2=T_2$ and $t_1\downarrow T_1$ and $t_1=T_0$ and $t_2\uparrow T_1$, respectively, to get
\begin{align*}
\int_{T_1}^{T_2}\int L(b(s,x))dm_s(x)ds&\geq \int u(T_1^+,x)d\mu_1(x)-\int u_2(x)d\mu_2(x),\\
\int_{T_0}^{T_1}\int L(b(s,x))dm_s(x)ds&\geq \int u(T_0,x)d\mu_0(x) -\int u(T_1^-,x)d\mu_1(x).
\end{align*}
Adding these two relations yields
\begin{align*}
\int_{T_0}^{T_2}\int L(b(s,x))dm_s(x)ds&\geq \int u(T_0)d\mu_0 -\int u(T_2)d\mu_2+\int (u(T_1^+)-u(T_1^-))d\mu_1\\
&\geq \int u(T_0)d\mu_0 -\int u_2d\mu_2+\int u_1d\mu_1,
\end{align*}
hence $V(T_0,\mu,\mu_1,\mu_2)\geq D.$
\vspace{1mm}

\noindent
To prove the opposite inequality, let $u$ be a viscosity solution of \eqref{HJmain}. Due to Lemma \ref{approx}, there exists a uniformly bounded sequence of smooth super-solutions $u_n$ of \eqref{HJmain} (for possibly different $u_1,u_2$) such that $u_n\xrightarrow{n\rightarrow \infty} u$. By the bounded convergence theorem, this implies
\begin{align*}
    \lim_{n\rightarrow \infty}\bigg( \int u_n(T_0)d\mu_0 -\int u_n(T_2)d\mu_2- \text{\bf{1}}_{[0,T_1]}&(T_0)\int(u_n(T_1^{-})-u_n(T_1^{+}))d\mu_1 \bigg) \\
    &=\int u(T_0)d\mu_0 -\int u_2d\mu_2- \text{\bf{1}}_{[0,T_1]}(T_0)\int u_1d\mu_1,
\end{align*}
hence 
\begin{align*}
    D\geq \int u(T_0)d\mu_0 -\int u_2d\mu_2- \text{\bf{1}}_{[0,T_1]}(T_0)\int u_1d\mu_1,
\end{align*}
Taking supremum over $u$ and using \eqref{duality} we get $D\geq V(T_0,\mu_0,\mu_1,\mu_2)$, which finishes the proof.
\end{proof}
\vspace{4mm}

\noindent
Finally, we prove Corollary \ref{cor1}.

\begin{proof}
We assume that $t_0<T_1$, the other case being similar. Note that if $u$ satisfies \eqref{HJmain} in the classical sense and $m$ satisfies \eqref{FP}, then if $t_1<t_2<T_1$ or $T_1<t_1<t_2$ we have
\be\label{test}
\int u(t_2)dm_{t_2}-\int u(t_1)dm_{t_1}=\int_{t_1}^{t_2}\int H\left( \frac{D^2u}{2}\right)dm_s ds+\int_{t_1}^{t_2}\int \frac{D^2u}{2}b(s,x)\; dm_s(x)ds.
\ee
We first assume that $b=-H'\left( \frac{D^2u}{2}\right)$. Then, \eqref{test} becomes
$$
\int u(t_2)dm_{t_2}-\int u(t_1)dm_{t_1}=-\int_{t_1}^{t_2}\int L(b(s,x))dm_s(x)ds.
$$
Therefore, we deduce
\begin{align*}
\int_{T_1}^{T_2}\int L(b(s,x))dm_s(x)ds&=\int u(T_1^+)d\mu_1 -\int u_2 d\mu_2\\
\int_{T_0}^{T_1}\int L(b(s,x))dm_s(x)ds&=\int u(T_0)d\mu_0 -\int u(T_1^-)d\mu_1.
\end{align*}
We add the last two relations and since $u(T_1^+)-u(T_1^-)=-u_1$, it follows that 
$$\int_{T_0}^{T_2}\int L(b(s,x))dm_s(x)ds= \int u(T_0)d\mu_0 -\int u_1d\mu_1 -\int u_2d\mu_2,$$
which means that the supremum in \eqref{duality} is achieved.
\vspace{1mm}

\noindent
For the other direction we assume that the supremum in \eqref{duality} is achieved at $u_1^*,u_2^*\in C^2_b$. We may note that, due to standard regularity results and the strong ellipticity assumption on $H$, the solution $u(t,x)$ of \eqref{HJmain} is classical when $t\in [T_0,T_1)$ and $t\in (T_1,T_2]$. Then, again by \eqref{test}, by using the inequality $H(\frac{D^2u}{2})+\frac{D^2u}{2}b\geq -L(b)$ we have
\begin{align}
\int_{T_1}^{T_2}\int L(b(s,x))dm_s(x)ds&\geq \int u(T_1^+)d\mu_1 -\int u_2^* d\mu_2 \label{3.1}\\
\int_{t_0}^{T_1}\int L(b(s,x))dm_s(x)ds&\geq \int u(T_0)d\mu_0 -\int u(T_1^-)d\mu_1.\label{3.2}
\end{align}
We add \eqref{3.1} and \eqref{3.2} to deduce
$$\int_{T_0}^{T_2}\int L(b(s,x))dm_s(x)ds \geq  \int u(T_0)d\mu_0 -\int u_1^*d\mu_1 -\int u_2^*d\mu_2.$$
However, since the supremum in \eqref{duality} is achieved, the last holds as an equality, hence \eqref{3.1} and \eqref{3.2} are equalities as well, which implies that $H(\frac{D^2u(s,x)}{2})+\frac{D^2u(s,x)}{2}b(s,x)= -L(b(s,x))$. It is now straightforward to show that $b(s,x)=-H'\left(\frac{D^2u(s,x)}{2}\right)$.
\end{proof}

\section{Martingale Schr\"odinger bridges, Theorems \ref{main4}, \ref{main5}}

\noindent
In this section, we give the proofs of Theorems \ref{main4} and \ref{main5}.

\begin{proof}[Proof of Theorem \ref{main4}]
We denote ${\bf W}$ the 2-dimensional Brownian motion $(W,W^{\perp})$. Let $\mathbb{P}\ll \mathbb{P}_0$ be a measure in $\mathcal{C}_{t_0}(\mu_1,\mu_2)$. Then, by Girsanov's theorem there exists an adapted process $\bm \alpha_t=(\alpha^1_t,\alpha^2_t)$ such that, for any $t\in [t_0,T_2]$, $\frac{d\mathbb{P}}{d\mathbb{P}_0}\bigg|_{\mathcal{F}_t}=\text{exp}\left( \int_{t_0}^t \bm\alpha_sd {\bf W}_s -\frac{1}{2}\int_{t_0}^t |\bm \alpha_s|^2ds\right)$ with $\mathbb{E}^{\mathbb{P}}\left[ \int_{t_0}^{T_2}|\bm \alpha _t|^2dt\right]<\infty$. Since $X_t$ is $\mathbb{P}_0$-martingale, again by Girsanov $X_t-\int_{t_0}^t\alpha^1_s\sigma(X_s,Y_s)ds$ is a $\mathbb{P}$-martingale. However, $\mathbb{P}\in\mathcal{C}(\mu_1,\mu_2)$, hence $X_t$ is a $\mathbb{P}$-martingale, which implies that $\int_{t_0}^t \alpha^1_s \sigma(X_s,Y_s)ds$ is a $\mathbb{P}$-martingale. We can now easily deduce that $\alpha^1=0, \mathbb{P}$-a.s and, therefore, again by Girsanov's theorem, the dynamics \eqref{SVM} can be written as
\be \label{newSDE}
\begin{cases}
dX_t=\sigma(X_t,Y_t)d\tilde{W}_t,\\
dY_t=\left(b(X_t,Y_t)+ \tau_2(X_t,Y_t)\alpha^2_t\right)dt+\tau_1(X_t,Y_t)d\tilde{W}_t+\tau_2(X_t,Y_t)d\tilde{W}_t^{\perp},
\end{cases}\ee
where $\tilde{W}, \tilde{W}^{\perp}$ are independent $\mathbb{P}$-Brownian motions. We have
\begin{align*}
    H(\mathbb{P}|\mathbb{P}_0)=\mathbb{E}^{\mathbb{P}}\left[\int_{t_0}^{T_2}\alpha^2dW_t^{\perp}-\frac{1}{2}\int_{t_0}^{T_2}|\alpha^2_s|^2ds\right]=\frac{1}{2}\int_{t_0}^{T_2}\mathbb{E}^{\mathbb{P}}\left[ |\alpha^2_s|^2\right]ds.\label{entropy1}
\end{align*}
Furthermore, by the mimicking theorem \cite[Corollary 3.6]{brunick2013mimicking}, there exists a measurable function given by $\alpha_2(t,x,y):=\mathbb{E}^{\mathbb{P}}[\alpha^2_t|X_x=,Y_t=y]$ such that for any $t\in [0,T_2]$, $\mathcal{L}^{\mathbb{P}}(X_t,Y_t)=m_t:=\mathcal{L}^{\mathbb{P}}(\tilde{X}_t,\tilde{Y}_t)$, where $(X_t,Y_t)$ satisfies \eqref{newSDE} and $(\tilde{X}_t,\tilde{Y}_t)$ satisfies
\be \label{newSDEm}
\begin{cases}
d\tilde{X}_t=\sigma(\tilde{X}_t,\tilde{Y}_t)d\tilde{W}_t,\\
d\tilde{Y}_t=\left(b(\tilde{X}_t,\tilde{Y}_t)+ \tau_2(\tilde{X}_t,\tilde{Y}_t)\alpha_2(t,\tilde{X}_t,\tilde{Y}_t)\right)dt+\tau_1(\tilde{X}_t,\tilde{Y}_t)d\tilde{W}_t+\tau_2(\tilde{X}_t,\tilde{Y}_t)d\tilde{W}_t^{\perp}.
\end{cases}\ee
Thus, the left side of \eqref{dual} can be written as
\begin{align}
\inf_{\mathbb{P}\in\mathcal{C}_{t_0}(\mu_1,\mu_2)}&H(\mathbb{P}|\mathbb{P}_0)=\inf_{\substack{\alpha_2\in L^2(dm_t\otimes dt)\\ X_{T_1}\sim \mu_1, \;X_{T_2}\sim \mu_2}}\bigg\{ \frac{1}{2}\int_{t_0}^{T_2}\mathbb{E}^{\mathbb{P}}\left[ |\alpha_2(t,\tilde{X}_t,\tilde{Y}_t)|^2\right]dt\bigg\}\nonumber\\
&=\inf_{\alpha_2\in L^2(dm_t\otimes dt)}\sup_{u_2, u_2\in Lip}\bigg\{ \frac{1}{2}\int_{t_0}^{T_2}\mathbb{E}^{\mathbb{P}}\left[ |\alpha_2(t,\tilde{X}_t,\tilde{Y}_t)|^2\right]dt+{\bf 1}_{[0,T_1]}(t_0)\int u_1dm_{T_1}\nonumber\\
&\hspace{6cm}+\int u_2dm_{T_2}-{\bf 1}_{[0,T_1]}(t_0)\int u_1d\mu_1-\int u_2 d\mu_2\bigg\}.\label{int1}
\end{align}
By Proposition \ref{comp2} and a similar argument as in the proof of Theorem \ref{main3} (Step 2),
we can check that the function
\begin{align*}
\mathcal{L}((W,(m_t)_{t\in[t_0,T]}),(u_1,u_2))=\frac{1}{2}\int_{t_0}^{T_2}\left|\frac{dW}{dm_t\otimes dt}\right|^2dm_tdt+{\bf 1}&_{[0,T_1]}(t_0)\int u_1dm_{T_1}+\int u_2dm_{T_2}\\
&-{\bf 1}_{[0,T_1]}(t_0)\int u_1d\mu_1-\int u_2 d\mu_2
\end{align*}
satisfies the conditions of the Von-Neumann theorem (Theorem \ref{Von}) and hence the infimum in \eqref{int1} is a minimum and we can change the order of the infimum and the supremum to find
\begin{align*}
\inf_{\mathbb{P}\in\mathcal{C}_{t_0}(\mu_1,\mu_2)}H(\mathbb{P}|\mathbb{P}_0)=\sup_{u_1,u_2\in Lip}\bigg\{ u(t_0,X_{t_0},Y_{t_0})-\int u_2d\mu_2-{\bf 1}_{[0,T_1]}(t_0)\int u_1d\mu_1\bigg\}\label{dual3},
\end{align*}
where $u(t,x,y),\; t\in [0,T_2]$ is given by
\be\label{mfcu}
u(t,x,y)=\inf_{\alpha_2\in L^2(dm_t\otimes dt)}\bigg\{ \frac{1}{2}\int_t^{T_2}\int |\alpha_2(s,x,y)|^2dm_s(x,y)ds+{\bf 1}_{[0,T_1]}(t)\int u_1dm_{T_1}+\int u_2 dm_{T_2}\bigg\},
\ee
where $m_s$ is the law at times $s$ of the process $(X_s,Y_s)$ satisfying \eqref{newSDEm}, as before, with $(X_t,Y_t)=(x,y)$. We can now check that $u$ satisfies \eqref{HJS}, which finishes the proof of \eqref{dual}. Indeed, if $t\in (T_1,T_2]$, then by well known results of stochastic optimal control theory, $u$ satisfies \eqref{HJS} in the viscosity sense in that interval. If $t\in [0,T_1]$, then we break the integral in \eqref{mfcu} and minimize first over $(T_1,T_2]$ and then over $[t,T_1]$ to get
\begin{align*}
u(t,x,y)=\inf_{\substack{\alpha_2\in L^2(dm_s\otimes ds)\\ s\in [t,T_1]}}\bigg\{ \int_t^{T_1}\int \frac{|\alpha_2|^2}{2}dm_sds+\int (u_1(x)+u(T_1^+,x',y'))dm_{T_1}(x',y')\bigg\}.
\end{align*}
Once again, by standard results of stochastic optimal control theory (here the terminal condition is $u(T_1)=u_1+u(T_1^+)$), $u(t,\cdot,\cdot)$ satisfies \eqref{HJS} whenever $t\in [0,T_1)$.
\vspace{2mm}

\noindent
Finally, if the supremum in \eqref{dual} is achieved, we observe that due to Corollary \ref{reg1}, the solution $u^*(t,x)$ of \eqref{HJS} (with $u_1=u_1^*$ and $u_2=u_2^*$) is bounded and Lipschitz uniformly in $t\in [0,T_1)\cap (T_1,T_2]$. In particular, $\partial_y u^*(t,x,y)$ exists a.e. In addition, $\partial_y u^*(T_1,x,y)=\partial_y( u_1^*(x)+u^*(T_1^+,x,y))=\partial_y u^*(T_1^+,x,y)$, hence $\partial_y u^*(t,x,y)$ can be defined for all $t\in [0,T_2]$. To conclude, we note that the optimal $\alpha_2$ in \eqref{mfcu} is given by $\alpha_2(t,x,y)=-\tau_2(x,y)\partial_y u^*(t,x,y)$, which, combined with $\alpha^1\equiv 0$, implies \eqref{optimalP}.

\end{proof}

\begin{remark}
(i) If they exist, $u_1^*,u_2^*$ are called \textit{Schr\"odinger potentials}.
\vspace{1mm}

\noindent
(ii) Assume for simplicity that $t_0=0$. If the solution $u^*$ of \eqref{HJS} is classical, then we can write \eqref{optimalP} as
\be\label{newoptimalP}
\frac{d\mathbb{P}^*}{d\mathbb{P}_0}\bigg|_{\mathcal{F}_{T_2}}=e^{u^*(0,X_0,Y_0)-u_1^*(X_{T_1})-u_2^*(X_{T_2})-\int_0^{T_2} \Delta_t dW_t},\ee
where $\Delta_t=-\sigma(X_t,Y_t) \partial_x u^*(t,X_t,Y_t)-\tau_1(X_t,Y_t)\partial_y u^*(t,X_t,Y_t)$. Indeed, let $Z_t=(X_t,Y_t)$. We have by It\^o's formula, that whenever $0\leq t_1<t_2<T_1$ or $T_1<t_1<t_2\leq T_2$
\be\label {test2}
u^*(t_2,Z_{t_2})-u^*(t_1,Z_{t_1})=\frac{1}{2}\int_{t_1}^{t_2}\tau_2^2|\partial_yu^*|^2dt+\int_{t_1}^{t_2}\tau_2\partial_y u^*dW_t^{\perp}+\int_{t_1}^{t_2}(\sigma \partial_x u^*+\tau_1\partial_y u^*) dW_t.
\ee
We set $t_1=0, t_2\rightarrow T_1^-$ and $t_1\rightarrow T_1^+, t_2=T_2$ in \eqref{test2} and we add to discover
\begin{align*}
-u_2^*(X_{T_2})-u_1^*(X_{T_1})+\int_0^{T_2}(\sigma\partial_x u^*+\tau_1\partial_y u^*)dW_t+&u^*(0,Z_0)\\
&=-\int_0^{T_2}\tau_2 \partial_y u^*dW^{\perp}_t-\frac{1}{2}\int_0^{T_2}\tau_2^2|\partial_y u^*|^2dt,
\end{align*}
which gives \eqref{newoptimalP} from \eqref{optimalP}. In this case, $-u_1^*(X_{T_1})-u_2^*(X_{T_2})-\int_0^{T_2} \Delta_t dW_t$ is called \textit{Schr\"odinger portfolio}. We also notice that \eqref{newoptimalP} coincides with formula for the optimizer derived in \cite{guyon2022dispersion} in the case where there is no VIX constraint.
\vspace{1mm}

\noindent
(iii) Problem \eqref{minimization} is not always admissible, therefore in Theorem \ref{main4} the admissibility is included as an assumption. Indeed, if $\sigma(x,y):=x$ in \eqref{SVM}, then $X_t$ is a geometric Brownian motion and
for any $\mathbb{P}\in \mathcal{C}_{t_0}(\mu_1,\mu_2)$ such that $\mathbb{P}\ll \mathbb{P}_0$, the distributions $\mathcal{L}^{\mathbb{P}}(X_{T_1})$, $\mathcal{L}^{\mathbb{P}}(X_{T_2})$ do not change.

\end{remark}
\vspace{4mm}

\noindent
To show Theorem \ref{main5}, we start by proving two auxiliary results that will be useful in the proof.
\begin{proposition}\label{deltacont}
    Suppose that $b,\sigma$ satisfy (A3) and $\sigma(x,y)=x\tilde{\sigma}(y)$ for some bounded and Lipschitz function $\tilde{\sigma}:\R\rightarrow \R$. For $x,y>0$ and $\delta\in \R$, we consider the optimal control problem
\be\label{optimalc'}
\begin{split}
u(T_1,x,&\;y;\delta)\\
&=\inf_{\alpha}\bigg\{ \frac{1}{2}\int_{T_1}^{T_2}\int|\alpha(t,x,y)|^2dm_t(x,y)dt+\int u_2(x)dm_{T_2}(x)-\delta \int \log(x)dm_{T_2}(x)\bigg\},
\end{split}
\ee
where the infimum is taken over all square integrable controls $\alpha$ and $(m_t)_{t\in [T_1,T_2]}$ is the distribution of $(\tilde{X}_t,\tilde{Y}_t)$ satisfying 
\be \label{controlledSDE}
\begin{cases}
d\tilde{X}_t=\sigma(\tilde{X}_t,\tilde{Y}_t)d\tilde{W}_t,\\
d\tilde{Y}_t=\left( b(\tilde{X}_t,\tilde{Y}_t)+ \tau_2(\tilde{X}_t,\tilde{Y}_t)\alpha(t,\tilde{X}_t,\tilde{Y}_t)\right)dt+\tau_1(\tilde{X}_t,\tilde{Y}_t) d\tilde{W}_t+\tau_2(\tilde{X}_t,\tilde{Y}_t)d\tilde{W}_t^{\perp},\\
\tilde{X}_{T_1}=x,\;\;\tilde{Y}_{T_1}=y,
\end{cases}\ee
for $t\in [T_1,T_2]$ and $\tilde{W},\tilde{W}^{\perp}$ two independent Brownian motions.
Then, the map $\delta\mapsto u(T_1,x,y;\delta)$ is continuous.
\end{proposition}

\begin{proof}
Let $\delta_0\in \R$. We will show that $\delta\mapsto u(T_1,x,y;\delta)$ is continuous at $\delta_0$.
\vspace{1mm}

\noindent
Since this map is obtained as the infimum of a family of continuous in $\delta$ functions, we deduce that it is upper semicontinuous in $\delta$, hence
\be\label{uppersemi}
\limsup_{\delta\rightarrow \delta_0}u(T_1,x,y;\delta)\le u(T_1,x,y;\delta_0).
\ee
We will now show the lower semicontinuity. For $\delta\in \R$, we let $\alpha^{\delta}$ be the optimal control in \eqref{optimalc'} and we denote by $m_t^{\delta}$ the distribution of $(X_t,Y_t)$ when $\alpha=\alpha^{\delta}$.

\noindent
For $\delta\in \R$, by using $\alpha^{\delta}$ as a control we have
\begin{align}
u(T_1,x,y;\delta_0)-u(T_1,x,y;\delta)&\leq (\delta_0-\delta)\int -\log(x)dm_{T_2}^{\delta}(x,y)=(\delta_0-\delta)\mathbb{E}[-\log(\tilde{X}_{T_2})]\nonumber\\
&=(\delta_0-\delta)\left(-\log(x)+\mathbb{E}\left[\int_{T_1}^{T_2}\tilde{\sigma}^2(\tilde{Y}_t)dt\right] \right)\nonumber\\
&\le |\delta_0-\delta|\left(-\log(x)+(T_2-T_1)\|\tilde{\sigma}^2\|_{\infty}\right).\nonumber
\end{align}
By sending $\delta\rightarrow \delta_0$, 
we deduce
$$\liminf_{\delta\rightarrow \delta_0}u(T_1,x,y;\delta)\geq u(T_1,x,y;\delta_0),$$
which is the desired lower semicontinuity.
\end{proof}
\vspace{3mm}

\noindent
In the following Proposition we show that the minimum in \eqref{minvix} can be attained.
\begin{proposition}\label{martingality}
    Under the assumptions of Theorem \ref{main5}, there exists a $\mathbb{P}^*\in\mathcal{C}_{t_0}(\mu_1,\mu_2)$ such that $\mathcal{L}^{\mathbb{P}^*}(V_{\mathbb{P}^*})\le_{c,l}\mu_3$ and $D'_{\mathbb{P}_0}=H(\mathbb{P}^*|\mathbb{P}_0)$.
\end{proposition}

\begin{proof}
    Let $(\mathbb{P}_n)_{n\in\mathbb{N}}\in \mathcal{C}_{t_0}(\mu_1,\mu_2)$ such that $\mathcal{L}^{\mathbb{P}_n}(V_{\mathbb{P}_n})\le_{c}\mu_3$ and $H(\mathbb{P}_n|\mathbb{P}_0)\xrightarrow{n\rightarrow \infty}D'_{\mathbb{P}_0}$. Consider the probability densities $f_n:=\frac{d\mathbb{P}_n}{d\mathbb{P}_0}$. Then, the sequence $H(\mathbb{P}_n|\mathbb{P}_0)=\int f_n\log(f_n)d\mathbb{P}_0$ is bounded and hence, by the Dunford-Pettis theorem, there exists a probability density $f$ such that $f_n\rightharpoonup f$ weakly in $L^1$. We let $\mathbb{P}^*$ to be the probability measure such that $d\mathbb{P}^*= fd\mathbb{P}_0$. It is straightforward to check that $\mathcal{L}^{\mathbb{P}^*}(X_{T_1})=\mu_1$ and $\mathcal{L}^{\mathbb{P}^*}(X_{T_2})=\mu_2$.
\vspace{1mm}

    \noindent
    To show that $X_t$ is a $\mathbb{P}^*$-martingale, it suffices to show that $\mathbb{E}^{\mathbb{P}^*}[(X_t-X_s)\varphi]=0$ for any $T_2\ge t>s\ge t_0$ and $\varphi\in L^{\infty}(\mathcal{F}_s)$. Since $X_t$ is a $\mathbb{P}_n$-martingale, we know that 
    \be\label{eq00001}
    \mathbb{E}^{\mathbb{P}_0}[f_n(X_t-X_s)\varphi]=0.
    \ee
    However, by the assumptions on $\mu_2$, $X_{T_2}$ is uniformly integrable with respect to $(\mathbb{P}_n)_{n\in\mathbb{N}}$ and the la Valle\'e-Poussin theorem yields a convex function $\psi:\R\rightarrow \R_{\ge 0}$ with superlinear growth such that $\int \psi(x) d\mu_2(x)<+\infty$. Thus, due to the martingality of $X_t$, we have that for any $r\in [t_0,T_2]$
    $$\int \psi(X_r)d\mathbb{P}_n\leq \int \psi(X_{T_2})d\mathbb{P}_n=\int \psi(x)d\mu_2(x)\leq \infty$$
    and hence $X_r$ is uniformly integrable with respect to the family $(\mathbb{P}_n)_{n\in\mathbb{N}}$. Now Lemma \ref{condint} applies and we get $(X_t-X_s)f_n\rightharpoonup (X_t-X_s)f$ weakly in $L^1$. The result follows by sending $n\rightarrow \infty $ in \eqref{eq00001}.
\vspace{1mm}

\noindent
To show that $\mathcal{L}^{\mathbb{P}^*}(V_{\mathbb{P}^*})\le_{c,l}\mu_3$, it suffices to show 
\be\label{eq000002}
\int h\left( \mathbb{E}^{\mathbb{P}^*}[-\log(X_{T_2})+\log(X_{T_1})|\mathcal{F}_{T_1}]\right)d\mathbb{P}^*\leq \int h(x)d\mu_3(x),
\ee
for any linear and convex and lower bounded function $h:[0,+\infty)\rightarrow \R$. We set $V=-\log(X_{T_2})+\log(X_{T_1})$. By using the standard formula $\mathbb{E}^{\mathbb{P}_n}[V|\mathcal{F}_{T_1}]=\frac{\mathbb{E}^{\mathbb{P}_0}[Vf_n|\mathcal{F}_{T_1}]}{\mathbb{E}^{\mathbb{P}_0}[f_n|\mathcal{F}_{T_1}]}$, we have by the properties of $(\mathbb{P}_n)_{n\in\mathbb{N}}$
\begin{align}
    \int h(x)d\mu_3(x)\ge \mathbb{E}^{\mathbb{P}_n}\left[ h\left( \mathbb{E}^{\mathbb{P}_n}[V|\mathcal{F}_{T_1}] \right) \right]&=\int h\left( \frac{\mathbb{E}^{\mathbb{P}_0}[Vf_n|\mathcal{F}_{T_1}]}{\mathbb{E}^{\mathbb{P}_0}[f_n|\mathcal{F}_{T_1}]}\right) \mathbb{E}^{\mathbb{P}_0}[f_n|\mathcal{F}_{T_1}] d\mathbb{P}_0\nonumber\\
    &:=\Psi\left(\mathbb{E}^{\mathbb{P}_0}[Vf_n|\mathcal{F}_{T_1}],\mathbb{E}^{\mathbb{P}_0}[f_n|\mathcal{F}_{T_1}]\right),\label{eq0002}
\end{align}
where $\Psi: L^1\times L^1\rightarrow \R$ is the function defined as
\be\nonumber
\Psi(a,b)=\int h\left(\frac{a}{b} \right) b\; d\mathbb{P}_0.
\ee
Since $h$ is convex and lower bounded, standard arguments imply that $\Psi$ is convex and lower semi-continuous, therefore it is also weakly lower semi-continuous.
We now observe that, due to our assumptions on $\mu_1,\mu_2$, $V$ is uniformly integrable with respect to $(\mathbb{P}_n)_{n\in\mathbb{N}}$, hence Lemma \ref{condint}(iii) gives $\mathbb{E}^{\mathbb{P}_0}[f_n|\mathcal{F}_{T_1}]\rightharpoonup \mathbb{E}^{\mathbb{P}_0}[f|\mathcal{F}_{T_1}]$ and $\mathbb{E}^{\mathbb{P}_0}[Vf_n|\mathcal{F}_{T_1}]\rightharpoonup \mathbb{E}^{\mathbb{P}_0}[Vf|\mathcal{F}_{T_1}]$ weakly in $L^1$. 
With this in mind, passing to the limit in \eqref{eq0002} yields
$$\int h(x)d\mu_3(x)\geq \Psi\left(\mathbb{E}^{\mathbb{P}_0}[Vf|\mathcal{F}_{T_1}],\mathbb{E}^{\mathbb{P}_0}[f|\mathcal{F}_{T_1}]\right),$$
which is \eqref{eq000002}.
\end{proof}
\vspace{3mm}

\noindent
We are now ready to show Theorem \ref{main5}. 
\vspace{2mm}

\begin{proof}[Proof of Theorem \ref{main5}]
The proof follows the lines of the proof of Theorem \ref{main4} but we have to give special treatment to the terms involving the VIX constraint $\mathcal{L}^{\mathbb{P}}(V'_{\mathbb{P}})\leq \mu_3$. We will be using the notation $\mathcal{P}'=\{ \mathbb{P}\in\mathcal{P}:\; X_t\text{ is a }\mathbb{P}\text{-martingale}\}$. We have 
\begin{align}
D'_{\mathbb{P}_0}&=\inf_{\mathbb{P}\in\mathcal{C}_{t_0}(\mu_1,\mu_2), \text{VIX}(G)}\sup_{u_1,u_2,u_3}\bigg\{ H(\mathbb{P}|\mathbb{P}_0)+\mathbb{E}^{\mathbb{P}}[u_2(X_{T_2})+u_1(X_{T_1})+u_3(V'_{\mathbb{P}})]-\sum_{i=1}^3\int u_i(x)\mu_i(dx)\bigg\}\nonumber\\
&=\inf_{\mathbb{P}\in\mathcal{P}', d\mathbb{P}=fd\mathbb{P}_0}\sup_{u_1,u_2,u_3}\bigg\{ \int f\log(f)d\mathbb{P}_0+\int u_3\left(\frac{\mathbb{E}^{\mathbb{P}_0}[Wf|\mathcal{F}_{T_1}]}{\mathbb{E}^{\mathbb{P}_0}[f|\mathcal{F}_{T_1}]}\right)\mathbb{E}^{\mathbb{P}_0}[f|\mathcal{F}_{T_1}]d\mathbb{P}_0\nonumber\\
&\hspace{6cm}+\mathbb{E}^{\mathbb{P}_0}[u_2(X_{T_2})f]+\mathbb{E}^{\mathbb{P}_0}[u_1(X_{T_1})f]-\sum_{i=1}^3\int u_i(x)\mu_i(dx)\bigg\},\label{equation1001}
\end{align}
where $W=-\log(X_{T_2})+\log(X_{T_1})$ and the supremum is taken over all bounded and continuous $u_1$, $u_2$, and $u_3$ convex, continuous and lower bounded. We notice that the functional inside the brackets (which we name $\mathcal{G}$), as a function of $f$ and $u_1,u_2,u_3$, is convex in $f\in L^1(\Omega)$, linear in $u_1,u_2,u_3$ and weakly (in $L^1$) lower semicontinuous with respect to $f$. In addition, for any fixed $u_1,u_2,u_3$ and constant $C$, the set $A=\{f\in L^1: f\ge 0, \; \int fd\mathbb{P}_0=1,\; \mathcal{G}(f, u_1,u_2,u_3)\leq C\}$ is compact with respect to the weak topology of $L^1$.
\vspace{1mm}

\noindent
Indeed, suppose that the sequence $(f_n)_{n\in\mathbb{N}}\in A$. Since $u_1,u_2,u_3$ are lower bounded, $\mathcal{G}(f,u_1,u_2,u_3)\leq C$ gives an upper bound for $\int f_n\log(f_n)d\mathbb{P}_0$.
On the one hand, the Dunford-Pettis theorem implies that $(f_n)_{n\in\mathbb{N}}$ converges weakly in $L^1$ (up to subsequence) to some $f\in L^1(\Omega)$. In addition, due to the assumption on $\tilde{\sigma}$, it is easy to prove that $\mathbb{E}^{\mathbb{P}_0}[X_t^2f_n]$ are uniformly bounded. Therefore, due to the argument from Proposition \ref{martingality}, the measure $\mathbb{P}$ with $d\mathbb{P}=f\mathbb{P}_0$ remains in $\mathcal{P}'$. The compactness of the set $A$ follows.
\vspace{1mm}

\noindent
The conditions for the Von-Neumann theorem \ref{Von} are satisfied, thus we can write \eqref{equation1001} as
\begin{align}
D'_{\mathbb{P}_0}&=\sup_{u_1,u_2,u_3}\inf_{\mathbb{P}\in\mathcal{P}', d\mathbb{P}=fd\mathbb{P}_0}\bigg\{ \int f\log(f)d\mathbb{P}_0+\int u_3\left(\frac{\mathbb{E}^{\mathbb{P}_0}[Wf|\mathcal{F}_{T_1}]}{\mathbb{E}^{\mathbb{P}_0}[f|\mathcal{F}_{T_1}]}\right)\mathbb{E}^{\mathbb{P}_0}[f|\mathcal{F}_{T_1}]d\mathbb{P}_0\nonumber\\
    &\hspace{6cm}+\mathbb{E}^{\mathbb{P}_0}[u_2(X_{T_2})f]+\mathbb{E}^{\mathbb{P}_0}[u_1(X_{T_1})f]-\sum_{i=1}^3\int u_i(x)\mu_i(dx)\bigg\}.\label{equation1002}
\end{align}
Now, using the fact that $V'_{\mathbb{P}}\ge 0$ whenever $\mathbb{P}\in \mathcal{P}'$ and the equality 
$$\mathbb{E}^{\mathbb{P}}[u_3(V'_{\mathbb{P}})]=\inf_{V\in\mathcal{F}_{T_1}}\sup_{\Delta\in L^{\infty}(\mathcal{F}_{T_1})}\left\{ \mathbb{E}^{\mathbb{P}}[u_3(V)+\Delta(W-V)]\right\},$$
\eqref{equation1002} can be written as
\begin{align}
D'_{\mathbb{P}_0}&=\sup_{u_1,u_2,u_3}\inf_{\mathbb{P}\in\mathcal{P}',\; d\mathbb{P}=fd\mathbb{P}_0}\bigg\{ \int f\log(f)d\mathbb{P}_0+\mathbb{E}^{\mathbb{P}_0}[u_2(X_{T_2})f]+\mathbb{E}^{\mathbb{P}_0}[u_1(X_{T_1})f]\nonumber\\  
&\hspace{3.8cm}-\sum_{i=1}^3\int u_i(x)\mu_i(dx)+\inf_{V\in\mathcal{F}_{T_1}}\sup_{\Delta\in L^{\infty}(\mathcal{F}_{T_1})}\left\{ \mathbb{E}^{\mathbb{P}}[u_3(|V|)+\Delta(W-|V|)]\right\}\bigg\}.\nonumber
\end{align}
With a similar Von-Neumann argument as previously, we can rewrite
\begin{align}
D'_{\mathbb{P}_0}&=\sup_{u_1,u_2,u_3}\inf_{V\in\mathcal{F}_{T_1}}\sup_{\Delta\in L^{\infty}(\mathcal{F}_{T_1})}\inf_{\mathbb{P}\in\mathcal{P}'}\bigg\{ H(\mathbb{P}|\mathbb{P}_0)+\mathbb{E}^{\mathbb{P}}[u_2(X_{T_2})]+\mathbb{E}^{\mathbb{P}}[u_1(X_{T_1})]+\mathbb{E}^{\mathbb{P}}[u_3(|V|)]\nonumber\\  
&\hspace{5cm}-\sum_{i=1}^3\int u_i(x)\mu_i(dx)+\mathbb{E}^{\mathbb{P}}\left[\Delta(-\log(X_{T_2})+\log(X_{T_1})-|V|)\right]\bigg\}.\label{equation1003}
\end{align}
By Girsanov's theorem, just as in the proof of Theorem \ref{main4}, for every $\mathbb{P}\in \mathcal{P}'$ there exists a process $\alpha_t$, such that for $t\in [t_0,T_2]$, $\frac{d\mathbb{P}}{d\mathbb{P}_0}\bigg|_{\mathcal{F}_t}=\text{exp}\left( \int_{t_0}^t\alpha_sdW_s^{\perp}-\frac{1}{2}\int_{t_0}^t|\alpha_s|^2ds\right)$,  $\mathbb{E}^{\mathbb{P}}\left[ \int_{t_0}^{T_2}|\alpha_t|^2dt\right]<\infty$ and (by the mimicking theorem \cite[Corollary 3.6]{brunick2013mimicking}) in the minimization problem $\alpha$ can be taken to be a function of $t,x,y$. Hence, \eqref{equation1003} becomes 
\begin{align}
D'_{\mathbb{P}_0}&=\sup_{u_1,u_2,u_3}\inf_{V\in\mathcal{F}_{T_1}}\sup_{\Delta\in L^\infty(\mathcal{F}_{T_1})}\inf_{\alpha\in L^2(dm_t\otimes dt)}\bigg\{\frac{1}{2}\int_{t_0}^{T_2}\mathbb{E}^{\mathbb{P}}\left[ |\alpha(t,\tilde{X}_t,\tilde{Y}_t)|^2\right]dt-\sum_{i=1}^3\int u_i(x)\mu_i(dx)\nonumber\\
&\hspace{2.2cm}+\mathbb{E}^{\mathbb{P}}[u_1(\tilde{X}_{T_1})+u_2(\tilde{X}_{T_2})+u_3(|V|)]+\mathbb{E}^{\mathbb{P}}\left[\Delta\left( -\log(\tilde{X}_{T_2})+\log(\tilde{X}_{T_1})-|V|\right)\right]\bigg\},\label{eq101}
\end{align}
where $m_t$ is the distribution of $(\tilde{X}_t,\tilde{Y}_t)$ which has the dynamics (under $\mathbb{P}$) 
\be \label{newSDE'}
\begin{cases}
d\tilde{X}_t=\sigma(\tilde{X}_t,\tilde{Y}_t)d\tilde{W}_t,\\
d\tilde{Y}_t=\left( b(\tilde{X}_t,\tilde{Y}_t)+ \tau_2(\tilde{X}_t,\tilde{Y}_t)\alpha(t,\tilde{X}_t,\tilde{Y}_t)\right)dt+\tau_1(\tilde{X}_t,\tilde{Y}_t) d\tilde{W}_t+\tau_2(\tilde{X}_t,\tilde{Y}_t)d\tilde{W}_t^{\perp},
\end{cases}\ee
for independent $\mathbb{P}$-Brownian motions $\tilde{W},\tilde{W}^{\perp}$. By optimizing in $[T_1,T_2]$ first and then using the minmax principle once again, \eqref{eq101} is equal to
\begin{align}
D'_{\mathbb{P}_0}&=\sup_{u_1,u_2,u_3}\inf_{\alpha\in L^2(dm_t\otimes dt)}\inf_{V\in \mathcal{F}_{T_1}}\sup_{\Delta\in L^{\infty}(\mathcal{F}_{T_1})}\bigg\{ \frac{1}{2}\int_{t_0}^{T_1}\int |\alpha(t,x,y)|^2dm_tdt+\int u_1 dm_{T_1}-\sum_{i=1}^3\int u_id\mu_i\nonumber\\
&\hspace{3cm}-\mathbb{E}^{\mathbb{P}}\left[ \Delta\left(-\log(\tilde{X}_{T_1})+|V|\right) -u_3(|V|)-u(T_1,\tilde{X}_{T_1},\tilde{Y}_{T_1}; \Delta)\right]\bigg\},\label{eq102}
\end{align}
where $u$ is the value function of the optimal control problem for $s\in [T_1,T_2]$
\begin{align}
u(s,x,y;\delta)=\inf_{\alpha}\bigg\{ \frac{1}{2}\int_{s}^{T_2}\int|\alpha(t,x,y)|^2dm_t(x,y)dt+\int u_2(x)dm_{T_2}(x,y)-\delta \int \log(x)dm_{T_2}(x,y)\bigg\}\label{optimalc}
\end{align}
and the minimization is happening over all drifts $\alpha$ and $m_t$ is the distribution of the pair $(\tilde{X}_t,\tilde{Y}_t)$ satisfying \eqref{newSDE'} with $(\tilde{X}_s,\tilde{Y}_s)=(x,y)$. We notice that for any $(s,x,y)\in [T_1,T_2]\times (0,+\infty)\times \R$,
\begin{align*}
    u(s,x,y;\delta)&\leq \|u_2\|_{\infty}+|\delta|\left|\mathbb{E}[\;\log(X_T)\;] \right|\leq \|u_2\|_{\infty}+|\delta||\log(x)|+\frac{|\delta|}{2}\left|\mathbb{E}\left[ \int_{s}^{T_2}\tilde{\sigma}^2(Y_s)ds]  \right]\right|\\
    &\leq \|u_2\|_{\infty}+\frac{|\delta|\|\tilde{\sigma}^2\|_{\infty}}{2}+|\delta||\log(x)|,
\end{align*}
where $(X_t,Y_t)$ was taken to be the solution of \eqref{newSDE'} with $\alpha=0$, and likewise
\begin{align*}
    u(s,x,y;\delta)&\ge \inf_{\alpha}\bigg\{ \int u_2(x)dm_{T_2}(x,y)-\delta \int \log(x)dm_{T_2}(x,y)\bigg\}\\
    &\ge \inf_{\alpha}\bigg\{ -\|u_2\|_{\infty}-|\delta|\left|\mathbb{E}\left[ \log(\tilde{X}_{T_2}) \right]\right|\bigg\}\ge -\|u_2\|_{\infty}-|\delta||\log(x)|-\frac{|\delta|\|\tilde{\sigma}^2\|_{\infty}}{2}.
\end{align*}
Therefore, $u$ has logarithmic growth: $|u(s,x,y;\delta)|\leq C(1+|\log(x)|)$.
\vspace{1mm}

\noindent
By standard arguments from stochastic optimal control theory, it is known that $u(t,x,y;\delta)$ satisfies (in the viscosity sense) in $(T_1,T_2)\times(0,+\infty)\times \R$ the state constraint problem
\be\label{firsteq}
\begin{cases}
-\partial_tu -\mathcal{L}^0_{x,y}u+\frac{1}{2}\tau_2^2(x,y)(\partial_y u)^2=0, &t\in (T_1,T_2),\\
u(T_2,x,y;\delta)=u_2(x)-\delta \log(x),\\
u(t,0,y;\delta)=\text{sign}(\delta)(+\infty),
\end{cases}\ee
where $\mathcal{L}^0_{x,y}$ was defined in \eqref{gen}. The value of $u(t,0,y;\delta)$ follows from the fact that if $X$ in \eqref{newSDE'} starts from $0$, then $X$ remains $0$.
\vspace{1mm}

\noindent
In fact, $u$ is the unique viscosity solution of \eqref{firsteq}. Indeed, by considering the change of variables $x=e^w,\;w\in\R$, we have that $u(t,x,y)=u(t,e^w,y)=:v(t,w,y)$ and $v$ satisfies 
\be\label{equationafter}
\begin{cases}
   -\partial_t v -\mathcal{L}^0_{w,y}v+\frac{1}{2}\tau_2^2(e^w,y)(\partial_yv)^2=0, &t\in(T_1,T_2),\;(w,y)\in\R^2,\\
   v(T_2,w,y)=u_2(e^w)-\delta w,& (w,y)\in \R^2,
\end{cases}
\ee
in the viscosity sense. Due to assumption (A3) and since $u_2(e^w)-\delta w$ has at most linear growth, \cite[Theorem 2.1, Example 2.1]{da2006uniqueness} implies that there is at most one viscosity solution of \eqref{equationafter} with at most quadratic growth. Since $v(t,w,y):=u(t,e^w,y)$ has at most quadratic growth (in fact linear), this means that $v$ is uniquely characterized by \eqref{equationafter}. Hence, $u$ is the unique viscosity solution of \eqref{optimalc} with at most logarithmic growth.
\vspace{1mm}

\noindent
Returning to \eqref{eq102} we see that there is the term
\begin{align}
\inf_{V\in\mathcal{F}_{T_1}}\sup_{\Delta\in L^\infty(\mathcal{F}_{T_1})}\bigg\{\mathbb{E}^{\mathbb{P}}\left[ -\Delta\left(-\log(\tilde{X}_{T_1})+|V|\right)+u(T_1,\tilde{X}_{T_1},\tilde{Y}_{T_1}; \Delta)+u_3(|V|)\right]\bigg\}.\label{term}
\end{align}
We observe that the function 
$$(\omega,\delta)\mapsto -\delta\left(-\log(\tilde{X}_{T_1}(\omega))+|V(\omega)|\right)+u(T_1,\tilde{X}_{T_1}(\omega),\tilde{Y}_{T_1}(\omega); \delta)+u_3(|V(\omega)|)$$
takes finite values, is continuous in $\delta$ (Proposition \ref{deltacont}) and measurable in $\omega$ (because $(x,y)\mapsto u(T_1,x,y;\delta)$ and $u_3$ are continuous functions), therefore by Proposition \ref{rockafellar} (case (i)), we can rewrite \eqref{term} as
$$\inf_{V\in\mathcal{F}_{T_1}}\mathbb{E}^{\mathbb{P}}\left[\sup_{\delta\in \R}\left\{ -\delta\left(-\log(\tilde{X}_{T_1})+|V|\right)+u(T_1,\tilde{X}_{T_1},\tilde{Y}_{T_1}; \delta)+u_3(|V|)\right\}\right].$$
Since the supremum of continuous functions is lower semicontinuous, we may apply Proposition \ref{rockafellar} (case (ii)) to further rewrite \eqref{term} as
\begin{align*}
    \mathbb{E}^{\mathbb{P}}\left[\inf_{v\geq 0}\sup_{\delta\in \R}\left\{ -\delta\left(-\log(\tilde{X}_{T_1})+v\right)+u(T_1,\tilde{X}_{T_1},\tilde{Y}_{T_1}; \delta)+u_3(v)\right\}\right]=:\mathbb{E}^{\mathbb{P}}[\Phi(\tilde{X}_{T_1},\tilde{Y}_{T_1})],
\end{align*}
where $\Phi(x,y)=\inf_{v\geq 0}\sup_{\delta\in \R}\left\{ -\delta\left(-\log(x)+v\right)+u(T_1,x,y; \delta)+u_3(v)\right\}.$
Consequently, \eqref{eq102} becomes
\begin{align*}
D'_{\mathbb{P}_0}&=\sup_{u_1,u_2,u_3}\inf_{\alpha\in L^2(dm_t\otimes dt)}\bigg\{ \frac{1}{2}\int_{t_0}^{T_1}\int |\alpha(t,x,y)|^2dm_t(x,y)dt+\int u_1(x) dm_{T_1}(x,y)\nonumber\\
&\hspace{5.5cm}+\mathbb{E}^{\mathbb{P}}\left[ \Phi(\tilde{X}_{T_1},\tilde{Y}_{T_1})\right]-\sum_{i=1}^3\int u_i(x)d\mu_i(x)\bigg\}\nonumber\\
&=\sup_{u_1,u_2,u_3}\bigg\{ u(t_0,X_{t_0},Y_{t_0})-\sum_{i=1}^3\int u_id\mu_i\bigg\},\label{equation1004}
\end{align*}

\noindent
where $u$ is the value function of the optimal control problem for $s\in [t_0,T_1]$
\be\label{optimalc2}
u(s,x,y)=\inf_{\alpha}\bigg\{ \int_{s}^{T_1}\frac{|\alpha(t,x',y')|^2}{2}dm_t(x',y')dt+\int\left(u_1(x')+\Phi(x',y')\right)dm_{T_1}(x',y')\bigg\},
\ee
with the infimum taken over all controls $\alpha$ and $m_t$ is the distribution of the pair $(\tilde{X}_t,\tilde{Y}_t)$ which satisfies \eqref{newSDE'} with initial condition at $t=t_0$. We now observe that the function $\Phi:(0,+\infty)\times \R\rightarrow \R$ is bounded. Indeed, \footnote{Note that $u(T_1,x,y;\delta)$ in $\Phi$ was defined in \eqref{optimalc}.}
\begin{align*}
&\Phi(x,y)\geq \inf_{v\geq 0}\{ u(T_1,x,y; 0)+u_3(v)\},\text{ which is lower bounded, and}\\
& \Phi(x,y)\leq \inf_{v\geq 0}\sup_{\delta\in \R}\left\{ -\delta(-\log(x)+v)+\int u_2(x)dm_{T_2}(x,y)+\delta\int -\log(x)dm_{T_2}(x,y)+u_3(v)\right\}\\
&\hspace{1.2cm} =\int u_2(x)dm_{T_2}(x,y)+u_3\left( -\int \log(x)dm_{T_2}(x,y)+\log(x)\right),
\end{align*}
where $m_{T_2}$ is the distribution of $(X_{T_2},Y_{T_2})$ from \eqref{newSDE'} (with $\alpha_t=0$, $X_{T_1}=x$ and $Y_{T_1}=y$), which is upper bounded, because $u_2$ is bounded and 
$$-\int\log(x)dm_{T_2}(x,y)+\log(x)=\mathbb{E}^{\mathbb{P}_0}\left[ -\log(X_{T_2}) \right]+\log(x)=\mathbb{E}^{\mathbb{P}_0}\left[ \int_{T_1}^{T_2}\tilde{\sigma}^2(Y_t)dt+\int_{T_1}^{T_2}\tilde{\sigma}(Y_t)dW_t  \right]$$ 
is bounded by our assumptions on $\tilde{\sigma}$. Therefore, \eqref{optimalc2} is finite and formula \eqref{dual101} follows.
\vspace{2mm}

\noindent
To prove the characterization \eqref{character}, when $\tau_2$ is constant, we note that Lemma \ref{continuity estimates lemma} gives that $\Phi$ is continuous. Since $\Phi$ is continuous and bounded, we deduce (\cite[Theorem 2.1]{da2006uniqueness}) that there exists a unique bounded viscosity solution $v$ of \eqref{HJvix}. It is straightforward to see that $v(t,w,y):=u(t,e^w,y)$ for any $(t,w,x)\in [t_0,T_1]\times \R\times\R$, hence we \eqref{character} holds.
\end{proof}

\appendix

\section{Important results form the literature}

\noindent
We state the Von-Neumann theorem a proof of which can be found in \cite{orrieri2019variational}.
\begin{theorem}\label{Von}
Let $\mathbb{A}$ and $\mathbb{B}$ be convex subsets of some vector spaces and suppose that $\mathbb{B}$ is endowed with some Hausdorff topology. Let $\mathcal{L}$ be a function satisfying:
\begin{align*}
&a\mapsto \mathcal{L}(a,b) \text{ is concave in }\mathbb{A}\text{ for every }b\in\mathbb{B},\\
&b\mapsto \mathcal{L}(a,b) \text{ is convex in }\mathbb{B}\text{ for every }a\in\mathbb{A}.
\end{align*}
Suppose also that there exists $a_*\in\mathbb{A}$ and $C_*>\sup_{a\in\mathbb{A}}\inf_{b\in\mathbb{B}}\mathcal{L}(a,b)$ such that:
\begin{align*}
    &\mathbb{B}_*:=\{b\in\mathbb{B}|\mathcal{L}(a_*,b)\leq C_*\}\text{ is not empty and compact in }\mathbb{B},\\
    & b\mapsto \mathcal{L}(a,b)\text{ is lower-semicontinuous in }\mathbb{B}_* \text{ for every }a\in\mathbb{A}.
\end{align*}
Then,
$$\min_{b\in\mathbb{B}}\sup_{a\in\mathbb{A}}\mathcal{L}(a,b)=\sup_{a\in\mathbb{A}}\inf_{b\in\mathbb{B}}\mathcal{L}(a,b),$$
where the fact that the infimum is a minimum is part of the theorem.
\end{theorem}

\noindent
We will also need the following Proposition that provides us with conditions under which we can change the order of $\inf /\sup$ with an integral.

\begin{proposition} \label{rockafellar}
    Let $(\Omega,\mathcal{F},\mathbb{P})$ be a probability space and $R=L^{\infty}\text{ or }L^1$ in $(\Omega,\mathcal{F})$. Suppose that $f:\Omega\times \R^n\rightarrow \R\cup\{+\infty,-\infty\}$ is a function that satisfies one of the following set of conditions:\\
    (i) $f$ takes values in $\R$ and $f(\omega,x)$ is measurable in $\omega$ and continuous in $x$.\\
    (ii) $f(\omega,x)$ is lower semicontinuous in $x$ and there exists a function $g=g(\omega,y,x):\Omega\times\R\times\R\rightarrow \R$ which is continuous in $(x,y)$ and measurable in $\omega$ such that $f(\omega,x)=\sup_{y\in\R} g(\omega,y,x)$.\\
    Then,
    \be \label{swap}
    \inf_{X\in R} \int_{\Omega} f(\omega, X(\omega))d\mathbb{P}(\omega)=\int_{\Omega}\inf_{x\in \R}f(\omega, x)d\mathbb{P}(\omega).
    \ee
    In the (i) case, the infimums can also be replaced by supremums.
\end{proposition}

\begin{proof}
    We start by assuming (i). In this case, the result follows from \cite[Theorem 3A p.185]{rockafellar2006integral}. To prove \eqref{swap} for supremums, we observe that due to \cite[Proposition 2C p.174]{rockafellar2006integral}, $-f$ satisfies the conditions of \cite[Theorem 3A p.185]{rockafellar2006integral}, therefore
    $$\inf_{X\in R} \int_{\Omega} -f(\omega, X(\omega))d\mathbb{P}(\omega)=\int_{\Omega}\inf_{x\in \R}\{-f(\omega, x)\}d\mathbb{P}(\omega).$$
    \vspace{1mm}

    \noindent
    Now we assume (ii). By \cite[Proposition 2C p.174]{rockafellar2006integral} and \cite[Proposition 2R p.180]{rockafellar2006integral}, $f$ satisfies the conditions of \cite[Theorem 3A p.185]{rockafellar2006integral}. The conclusion follows.
\end{proof}

\section{Technical Propositions}

\noindent
In this section, to increase the readability of the paper, we state and prove some propositions that require more computational, though elementary, arguments.

\begin{proposition}\label{VIXoptionprice}
    With the notation introduced in section 1, suppose that $\mathbb{P}\in \text{VIX'}$. Then \eqref{optionpricebound} holds for any $K\in [0,+\infty)$.
\end{proposition}

\begin{proof}
    Define, for $\e > 0$, the convex and lower bounded function $h: [0,+\infty)\rightarrow \R$
$$h^{\e}(x) = \big(\e x - c\sqrt{x} + K\big)_+ ,$$
where $c=100\sqrt{\frac{2}{T_2-T_1}}$. We have
$$\int h^{\e}(V_{\mathbb{P}})\, d\mathbb{P} \;\leq\; \int h^{\e}(x)\, d\mu_3(x).$$
Sending $\e \to 0$, we obtain by Fatou's lemma and $h^\e(x)\leq (K-c\sqrt{x})_+$
$$\int (K - \text{VIX}_{\mathbb{P}})_+ \, d\mathbb{P} \leq \int (K-c\sqrt{x})_+ \, d\mu_3(x),$$
where $\text{VIX}_{\mathbb{P}}$ was defined in Remark \ref{VIXrem}. The proof is complete.
\end{proof}

\begin{lemma}\label{infcont}
    Let $X,Y$ be two metric spaces and $f: X\times Y\rightarrow \R$ be a lower bounded and lower semi-continuous function such that $x\mapsto f(x,y)$ is continuous for every $y\in Y$. If $Y$ is assumed to be compact, then the function $g:X\rightarrow \R$ with $g(x):=\inf_{y\in Y} f(x,y)$ is continuous.
\end{lemma}

\begin{proof}
    Let $(x_n)_{n\in \mathbb{N}}$ be a sequence in $X$ such that $x_n\rightarrow x\in X$. We will show that $\lim_n g(x_{n})=g(x)$. Let's assume that there exists a subsequence $(x_{k_n})_{n\in\mathbb{N}}$ such that $\liminf_n g(x_{k_n})<\limsup_n g(x_{k_n})$. Since $f$ is lower semi-continuous and $Y$ is compact, for every $n\in\mathbb{N}$, there exists $y_{k_n}\in Y$ such that $g(x_{k_n})=f(x_{k_n},y_{k_n})$. We also pick $y\in Y$ such that $g(x)=f(x,y)$. Due to the compactness of $Y$, $(y_{k_n})_{n\in\mathbb{N}}$ has a convergent subsequence (still denoted by $(y_{k_n})_{n\in\mathbb{N}}$) with limit $y'$. We have $g(x_{k_n})=f(x_{k_n},y_{k_n})\leq f(x_{k_n},y)$, hence by letting $n\rightarrow +\infty$ the lower semi-continuity of $f$ and the continuity of $f(\cdot,y)$ yield
    $$f(x,y')\leq \liminf_n g(x_{k_n})<\limsup_n g(x_{k_n})\leq \lim_n f(x_{k_n},y)=f(x,y)=g(x).$$
    Since $f(x,y')\ge g(x)$, this is a contradiction. The proof is complete.   
\end{proof}
\vspace{4mm}

\noindent
In the following lemma, we show that if we further assume that $\tau_2$ is constant, then the function $\Phi$ appearing in Theorem \ref{main5} is continuous.

\begin{lemma}\label{continuity estimates lemma}
    Suppose that (A3) holds, $\tau_2:\R^2\rightarrow \R$ is constant and that $\sigma(x,y)=x\tilde{\sigma}(y)$ for some Lipschitz and bounded $\tilde{\sigma}:\R\rightarrow \R$. Assume that $x_1,x_2>0$ and $y_1,y_2\in \R$. Then, for any $\delta\in\R$ there exists a constant $C$ depending only on $x_2,y_2$ such that 
    \be\label{"Lip"estimate}
u(T_1,x_1,y_1;\delta)\leq u(T_1,x_2,y_2;\delta)+C(|\delta|+1)\left( |\log(x_1)-\log(x_2)|+|x_1-x_2|+|y_1-y_2| \right),
    \ee
where the function $u(T_1,\cdot,\cdot;\delta)$ is defined in \eqref{optimalc'} in Proposition \ref{deltacont}. Furthermore, the function $\Phi:(0,+\infty)\times \R\rightarrow \R$ with 
$$\Phi(x,y)=\inf_{v\ge 0}\sup_{\delta\in\R}\{ -\delta(-\log(x)+v)+u(T_1,x,y;\delta)+u_3(v)\}$$
is continuous, where $u_3:\R\rightarrow \R$ is a convex and lower bounded function.
\end{lemma}

\begin{proof}
\textit{Step 1.} (proof of \eqref{"Lip"estimate})\\
Due to standard stochastic optimal control arguments, we notice that for $x>0$ and $y\in \R$ $u(T_1,x,y;\delta)$ can be written as $u(T_1,x,y;\delta)=\inf_{\alpha} J(x,y,\alpha;\delta)$ where
\be\label{optimalc''}
J(x,y,\alpha;\delta)
= \mathbb{E}\left[\frac{1}{2}\int_{T_1}^{T_2}|\alpha_t|^2dt+ u_2(X_{T_2}^{\alpha})-\delta \log(X_{T_2}^{\alpha})\right],
\ee
and where the infimum is taken over all square integrable adapted processes $\alpha$ and $(X_t^{\alpha},Y_t^{\alpha})_{t\in[T_1,T_2]}$ satisfies 
\be \label{controlledSDE'}
\begin{cases}
dX_t=X_t\tilde{\sigma}(Y_t)dW_t,\\
dY_t=\left( b(X_t,Y_t)+ \tau_2\alpha_t\right)dt+\tau_1(X_t,Y_t) dW_t+\tau_2dW_t^{\perp},\\
X_{T_1}=x,\;\;Y_{T_1}=y.
\end{cases}\ee
Let $\alpha$ be an admissible control for \eqref{optimalc''}. For $i=1,2$, we denote by $(X_t^{i,\alpha},Y_t^{i,\alpha})_{t\in [T_1,T_2]}$ the solution of \eqref{controlledSDE'} with initial condition $(X_{T_1}^{i,\alpha},Y_{T_1}^{i,\alpha})=(x_i,y_i)$. Due to the Lipschitz continuity of $b,\tau_2,\tau_1,\tilde{\sigma}$ and the boundedness of $\tilde{\sigma}$, a standard Gr\"onwall argument yields
\begin{align}\label{initialest}
    \mathbb{E}[|X_{T_2}^{1,\alpha}-X_{T_2}^{2,\alpha}|^2+|Y_{T_2}^{1,\alpha}-Y_{T_2}^{2,\alpha}|^2]\leq C\left(|x_1-x_2|^2+|y_1-y_2|^2\right),
\end{align}
for some constant $C$ depending only on $b,\tau_1,\tau_2$ and $\tilde{\sigma}$. In addition, by using the SDEs for $X_t^{1,\alpha}$ and $X_t^{2,\alpha}$
\begin{align}\label{initialest2}
    \mathbb{E}[|-\log(X_{t_2}^{1,\alpha})+\log(X_{T_2}^{2,\alpha})|]&\leq |-\log(x_1)+\log(x_2)|+\frac{1}{2}\int_{T_1}^{T_2}\mathbb{E}[\tilde{\sigma}^2(Y_t^{1,\alpha})-\tilde{\sigma}^2(Y_t^{2,\alpha})]dt\nonumber\\
    &\leq C\left( |-\log(x_1)+\log(x_2)| +|y_1-y_2|+|x_1-x_2|\right)
\end{align}
for some constant $C$ depending only on $b,\tau_1,\tau_2$ and $\tilde{\sigma}$.
\vspace{1mm}

\noindent
Now let $\alpha$ be an $\e$-optimal control for $u(T_1,x_2,y_2;\delta)$. We have
\begin{align*}
    J(x_1,y_1,\alpha;\delta)-J(x_2,y_2,\alpha;\delta)\leq \mathbb{E}[u_2(X_{T_2}^{1,\alpha})-u_2(X_{T_2}^{2,\alpha})]+\delta\mathbb{E}[-\log(X_{T_2}^{1,\alpha})+\log(X_{T_2}^{2,\alpha})],
\end{align*}
hence by using \eqref{initialest}, \eqref{initialest2} and the properties of the control $\alpha$ we get 
\begin{align*}
    u(T_1,x_1,y_1;\delta)-u(T_1,x_2,y_2;\delta)\leq C(|\delta|+1)\left( |\log(x_1)-\log(x_2)|+|x_1-x_2|+|y_1-y_2| \right)+\e.
\end{align*}
\eqref{"Lip"estimate} follows by sending $\e$ to $0$.
\vspace{2mm}

\noindent
\textit{Step 2.} (continuity of $\Phi$)\\  Let $\alpha^{\delta}$ be the optimal control in \eqref{optimalc''}. Then,
\be\label{other form phi}
\Phi(x,y)=\inf_{v\ge 0}\sup_{\delta\in \R}\left\{ \delta\left( \mathbb{E}[-\log(X_{T_2}^{\alpha^{\delta}})]+\log(x)-v\right) +\mathbb{E}\left[\int_{T_1}^{T_2}\frac{|\alpha^{\delta}_t|^2}{2}dt+u_2(X_{T_2}^{\alpha})\right] +u_3(v)\right\}.
\ee
Since, $\mathbb{E}[-\log(X_{T_2}^{\alpha^{\delta}})]+\log(x)=\frac{1}{2}\int_{T_1}^{T_2}\mathbb{E}[\tilde{\sigma}^2(Y_t^\alpha)]dt\leq \frac{\|\tilde{\sigma}\|^2_{\infty}(T_2-T_1)}{2}$, the above supremum is $+\infty$ if $v>\frac{\|\tilde{\sigma}\|^2_{\infty}(T_2-T_1)}{2}$, hence the infimum in \eqref{other form phi} can be taken over all $v$ in the compact set $[0,\frac{\|\tilde{\sigma}\|^2_{\infty}(T_2-T_1)}{2}]$.
\vspace{1mm}

\noindent
Let $g(v,x,y):=\sup_{\delta\in\R}\{-\delta(-\log(x)+v)+u(T_1,x,y;\delta)\}+u_3(v)$. We observe that 
$$g(v,x,y)=L_{x,y}\left(v-\log(x)\right)+u_3(v),$$
where $L_{x,y}(v)=\sup_{\delta\in \R}\{ -\delta v+u(T_1,x,y;\delta)\}$ is the convex conjugate of $\delta\mapsto u(T_1,x,y;\delta)$. It is known that $v\mapsto L_{x,y}(v)$ is continuous for any $x>0,y\in\R$ and that $(x,y)\mapsto L_{x,y}(v-\log(x))$ is lower semi-continuous for any $v$ (supremum of a family of continuous functions). We will show that $(x,y)\mapsto L_{x,y}(v-\log(x))$ is also upper semi-continuous and then the fact that $\Phi(x,y)=\inf_v g(v,x,y)$ is continuous follows from Lemma \ref{infcont}.
\vspace{1mm}

\noindent
Let $x_1,x_2>0$ and $y_1,y_2\in \R$. To simplify the notation we set 
$$\omega(x_1,x_2,y_1,y_2)=C\left( |\log(x_1)-\log(x_2)|+|x_1-x_2|+|y_1-y_2| \right)$$
for the quantity appearing on the right hand side of \eqref{"Lip"estimate}. We have by \eqref{"Lip"estimate}
\begin{align*}
L_{x_1,y_1}(v-\log(x_1))&=\sup_{\delta\in \R}\{-\delta(-\log(x_1)+v)+u(T_1,x_1,y_1;\delta)\}\\
&\hspace{-1cm}\leq \sup_{\delta\in\R}\left\{-\delta(-\log(x_1)+v)+u(T_1,x_2,y_2;\delta)+|\delta|\omega(x_1,x_2,y_1,y_2)\right\}+\omega(x_1,x_2,y_1,y_2)\\
&\hspace{-2cm}=\max\left\{ L_{x_2,y_2}(-\log(x_1)+v+\omega(x_1,x_2,y_1,y_2)),  L_{x_2,y_2}(-\log(x_1)+v-\omega(x_1,x_2,y_1,y_2))\right\}\\
&\hspace{3cm}+\omega(x_1,x_2,y_1,y_2).
\end{align*}
We now let $(x_1,y_1)\rightarrow (x_2,y_2)$ and since $v\mapsto L_{x,y}(v)$ is continuous and $\omega(x_1,x_2,y_1,y_2)\rightarrow 0$ we derive
$$\limsup_{(x_1,y_1)\rightarrow (x_2,y_2)}L_{x_1,y_1}(v-\log(x_1))\leq L_{x_2,y_2}(-\log(x_2)+v),$$
which is the desired upper semi-continuity.
\end{proof}

\section{Hamilton-Jacobi Equations} \label{HJequation}

\noindent
In this section we introduce the definition and the basic properties of solutions of Hamilton-Jacobi equations with the presence of a Dirac delta function. More specifically, equations of the form 
\be\label{HJgeneral}
\begin{cases}
    -\partial_tu +\mathcal{H}(t,x,u,Du,D^2u)=\delta_{T_1}(t)u_1(x),& (t,x)\in [0,T_2]\times \R^d,\\
    u(T_2,x)=u_2(x), & x\in \R^d,
\end{cases}
\ee
where $d\in \mathbb{N}$ and $\mathcal{H}:[0,T_2]\times\R^d\times \R^d\times \R^{d\times d}\rightarrow \R$.

\begin{definition}\label{defvisc}
    Let $u_1,u_2:\R^d\rightarrow \R$ be Lipschitz functions, $0<T_1<T_2$, $t\in [0,T_2]$ and $u:[0,T_2]\times \R^d\rightarrow \R$.
    \vspace{1mm}

    \noindent
    (i) We say that $u$ is a viscosity solution of \eqref{HJmain} if $u$ satisfies 
    \begin{equation}\label{eq1}
    \begin{split}
    &\begin{cases}
        -\partial_t u+\mathcal{H}(t,x,u,Du,D^2u)=0, &(t,x)\in (T_1,T_2]\times \R^d\\
        u(T_2,x)=u_2(x), & x\in \R^d, 
    \end{cases} \quad\quad\text{ and }\\
     &\hspace{2cm}\begin{cases}
        -\partial_t u+\mathcal{H}(t,x,u,Du,D^2u)=0,& (t,x)\in [0,T_1)\times \R^d,\\
        u(T_1,x)=u_1(x)+u(T_1^+,x),& x\in \R^d ,    
    \end{cases}
    \end{split}
    \end{equation}
    in the viscosity sense, where $u(T_1^+,x)$ that appears in the second problem stands for the solution of the first problem at time $T_1$.
\vspace{2mm}

    \noindent
    (ii) We say that $u$ satisfies \eqref{HJmain} in the classical sense or that $u$ is a classical solution of \eqref{HJmain} if $u\in BV([0,T]; C^2_b(\R^d))$ and satisfies the equations in \eqref{eq1} from part (i) in the classical sense. 

\vspace{2mm}

    \noindent
    (iii) We say that $u$ is a classical supersolution of \eqref{HJgeneral} if $u\in BV([0,T]; C^2_b(\R^d))$ and the equalities in \eqref{eq1} are satisfied as inequalities: $\leq$.
\end{definition}

\noindent
Following the theory of viscosity solutions of second order Hamilton-Jacobi equations (see \cite{crandall1992user}), we know that there exists a unique viscosity solution of \eqref{HJgeneral} when $\mathcal{H}$ satisfies
\begin{itemize}
    \item $\mathcal{H}\in C([0,T_2]\times\R^d\times\R\times\R^d\times \R^{d\times d}_{\text{sym}})$.
    \item $\mathcal{H}(t,x,r,p,X)\leq \mathcal{H}(t,x,s,p,Y)$  whenever $r\leq s$ and $Y\leq X$.
    \item There exists an increasing function $\omega: [0,+\infty)\rightarrow [0,+\infty)$ such that $\omega(0)=0$ and 
    $$\mathcal{H}(t,y,r,\alpha(x-y),Y)-\mathcal{H}(t,x,r,\alpha(x-y),X)\leq \omega( \alpha|x-y|^2+|x-y|)$$
    for any $x,y\in \R^d$, $r\in \R$ and $t\in [0,T_2]$, where $\alpha>0$ and $X,Y\in S_d$ satisfy
    $$-3\alpha \begin{pmatrix}
I & 0 \\
0 & I 
\end{pmatrix}\leq \begin{pmatrix}
X & 0\\
0 & -Y
\end{pmatrix}\leq 3\alpha \begin{pmatrix}
I & -I\\
-I & I
\end{pmatrix}.$$
\end{itemize}

\begin{remark}
(i) It is clear that the Hamiltonian $\mathcal{H}$ is Theorem \ref{main2} satisfies these three conditions. The last one holds because $H$ is degenerate elliptic and the inequality between the matrices implies that $X\leq Y$, therefore we can choose $\omega\equiv 0$.
\vspace{1mm}

\noindent
(ii) The Hamiltonian in Theorem \ref{main4} also satisfies these three conditions by similar arguments.
\end{remark}
\vspace{2mm}

\noindent
It can be shown that every viscosity solution of \eqref{HJgeneral} can be approximated by smooth super-solutions of \eqref{HJgeneral}.

\begin{lemma}\label{approx}
Suppose that $\mathcal{H}$ satisfies the above properties and that it is convex in the last three variables. Let $u$ be a bounded viscosity solution of \eqref{HJgeneral} in the sense of Definition \ref{defvisc}. Then, there exists a sequence $(u_n)_{n\in \mathbb{N}}$ of smooth functions in $([0,T_1)\times \R^d)\cup ((T_1,T_2]\times\R^d)$ such that 
$u_n(T_1,\cdot)$ is smooth, 
$-\partial_t u_n+\mathcal{H}(t,x,u_n,Du_n,D^2u_n)\leq 0$ for $t\neq T_1$ and $u_n\xrightarrow{n\rightarrow \infty}u$ pointwise. Furthermore $\sup_{(t,x)\in [0,T_2]\times\R^d} |u_n(t,x)|\leq \sup_{(t,x)\in [0,T_2]\times\R^d} |u(t,x)|$, for every $n\in \mathbb{N}$.
\end{lemma}

\begin{proof}
    For $\e>0$, we consider the sup-convolution
\be\nonumber
u_{\e}(t,x)=\sup_{(s,y)\in [T_1,T_2]\times \R^d}\left\{ u(s,y)-\frac{1}{2\e}(|x-y|^2+|t-s|^2)\right\}.
\ee
  By \cite[Chapter 4, Theorem 10]{katzourakis2015mollification}, we know that $u_{\e}$ is twice differentiable almost everywhere, $u_{\e}\xrightarrow{\e\rightarrow 0}u$ locally uniformly and satisfies $-\partial_tu_{\e}+\mathcal{H}(t,x,u_{\e},Du_{\e},D^2u_{\e})\leq 0$ almost everywhere. Let $u_{\e}^{\e}$ be the mollifications of $u_{\e}$ against a standard $\e$-mollifier $\phi_{\e}$. Then, from the convexity of $\mathcal{H}$ we deduce
  $$-\partial_t u_{\e}^{\e}+\mathcal{H}(t,x,u_{\e}^{\e},Du_{\e}^{\e},D^2u_{\e}^{\e})\leq 0.$$
  It is straightforward to show that $u_{\e}^{\e}(t,x)\xrightarrow{\e\rightarrow 0}u(t,x)$ for every $(t,x)\in [T_1^+,T_2]\times\R^d $, hence the sequence $(u_{1/n}^{1/n})_{n\in\mathbb{N}}$ has the desired properties in $[T_1^+,T_2]$. We repeat the same procedure (sup-convolution and mollification) for the convergence of $u_{1/n}^{1/n}$ in $[0,T_1]\times\R^d$. We omit the details.
\end{proof}

\subsection{A regularity result for a special case}
We consider the Hamilton-Jacobi equation \eqref{HJgeneral} with conditions $u_1, u_2:\R^d\rightarrow \R$ and with Hamiltonian $\mathcal{H}:[0,T]\times\R^d\times \R\times \R^d\times \R^{d\times d}\rightarrow \R$
\be\label{ham}
\mathcal{H}(t,x,u,p,q)=-\text{tr}(\sigma_0(x)\sigma^{\top}_0(x)q)+\frac{1}{2}\tau_0^2(x)|p|^2+b_0(x)\cdot p, \ee
where $\sigma_0:\R^d\rightarrow \R^{d\times d}$, $\tau_0: \R^d\rightarrow \R_{\ge0}$ and $b_0:\R^d\rightarrow \R^d$ are given functions. The following proposition holds for quasilinear parabolic equations with possibly degenerate second order term.

\begin{proposition}\label{regprop1}
Assume that $\mathcal{H}$ is as in \eqref{ham}, $\sigma_0\in C^{1,1}$ and $\tau_0,b_0\in C^2$ such that $|Db_0|+|D\tau_0|\leq \lambda_0 |\tau_0|$ for some constant $\lambda_0>0$. Let $T>0$, $g:\R^d\rightarrow \R$ a bounded Lipschitz function. Then, there exists a  unique viscosity solution of
$$\begin{cases}
    -\partial_t u +\mathcal{H}(x,Du,D^2u)=0,& (t,x)\in [0,T]\times \R^d,\\
    u(T,x)=g(x) ,& x\in \R^d.
\end{cases}$$ 
Furthermore, there exists a constant $C>0$ such that the bound $\|u(t,\cdot)\|_{L^{\infty}}\leq C\|g\|_{L^{\infty}}$ and the gradient bound $\|Du(t,\cdot)\|_{L^{\infty}}\leq C(T-t+\|Dg\|_{L^{\infty}})$ hold.
\end{proposition}

\begin{proof}
We perturb $a:=\sigma_0\sigma^{\top}_0$ to $a_{\e}=a+\e I$ and we regularize $g$ by using a standard mollifier $\rho_{\e}$: $g_{\e}=g*\rho_{\e}$. Let $u^{\e}$ be the solution of the (uniformly) parabolic equation 
$$\begin{cases}
    -\partial_t u^\e -\text{tr}(a_{\e} D^2u^{\e}) +\frac{1}{2}\tau_0^2|Du^{\e}|^2+b_0\cdot Du^\e=0,& (t,x)\in [0,T]\times \R^d\\
    u^{\e}(T,x)=g_{\e}(x), & x\in \R^d.
\end{cases}$$
By Bernstein's method, as in \cite[Lemma 2.2]{cardaliaguet2025mean}, the functions $(u^{\e})_{\e\in (0,1)}$ are uniformly bounded and uniformly Lipschitz, independently of $\e$. The stability property of viscosity solutions imply that $u^{\e}$ converges locally uniformly in $[0,T]\times \R^d$ to a function $u$ which is the unique viscosity solution of \eqref{HJgeneral} (with $\mathcal{H}$ as in \eqref{ham}). The proof is complete.
\end{proof}

\noindent
As a corollary, we obtain a regularity result for \eqref{HJgeneral} when $\mathcal{H}$ has the form \eqref{ham}.

\begin{corollary}\label{reg1}Assume that $\mathcal{H}$ is as in \eqref{ham}, $\sigma_0\in C^{1,1}$ and $\tau_0,b_0\in C^2$ is such that $|Db_0|+|D\tau_0|\leq \lambda_0 |\tau_0|$ for some constant $\lambda_0>0$. Furthermore, $u_1,u_2:\R^d\rightarrow \R$ are bounded and Lipschitz. Let $u$ be the unique viscosity solution to equation \eqref{HJgeneral} when \eqref{ham} holds. Then, $u$ is bounded and $u(t,\cdot)$ is Lipschitz uniformly in $t$.\end{corollary}

\begin{proof}
By Proposition \eqref{regprop1} on $(T_1,T_2]$, we obtain that $u(t,\cdot)$ is bounded and uniformly Lipschitz when $t\in (T_1,T_2]$. In particular, this implies that $u_1(x)+u(T_1^+,x)$ is bounded and Lipschitz, therefore by Proposition \ref{regprop1} once again, $u(t,\cdot)$ is bounded and uniformly Lipschitz when $t\in [0,T_1)$.
\end{proof}

\vspace{3mm}

\noindent
\textbf{Acknowledgments.}
The author wishes to thank P. E. Souganidis for his support and for many fruitful discussions during the course of this work. This work was written while the author was a Ph.D. student at the University of Chicago and was partially supported by P. E. Souganidis’ NSF grant DMS-1900599, ONR grant N000141712095, and AFOSR grant FA9550-18-1-0494.

\bibliographystyle{alpha}		
\bibliography{MOT}

\end{document}